\documentclass[11pt]{amsart}
\usepackage{amssymb,amsfonts,amsthm,amscd,stmaryrd,dsfont,esint,upgreek}
\usepackage[numbers]{natbib}
\usepackage[mathcal]{euscript}
\usepackage[usenames]{color}
\usepackage[top=1.25in,bottom=1.25in, left=1.25in,right=1.25in]{geometry}

\theoremstyle{plain}
\newtheorem{thm}{Theorem}
\newtheorem{cor}[thm]{Corollary}
\newtheorem{lem}{Lemma}
\newtheorem{prop}[lem]{Proposition}

\theoremstyle{definition}

\newtheorem{remark}[lem]{Remark}
\newtheorem{exam}[lem]{Example}

\usepackage{color} 
\usepackage[usenames,dvipsnames]{xcolor}

\definecolor{orange}{RGB}{255,127,0}

\numberwithin{lem}{section}
\numberwithin{equation}{section}

\newcommand{\R}{\mathbb{R}}

\newcommand{\Q}{\mathbb{Q}}

\newcommand{\N}{\mathbb{N}}

\newcommand{\Sy}{\mathbb{S}^\d}

\renewcommand{\d}{d}
\newcommand{\m}{q}

\newcommand{\p}{\alpha}
\newcommand{\Rd}{\mathbb R^\d}
\newcommand{\Qd}{\mathbb Q^\d}
\newcommand{\Zd}{\mathbb Z^\d}
\newcommand{\ep}{\varepsilon}

\newcommand{\barH}{\overline{H}}

\newcommand{\barL}{\overline{L}}
\newcommand{\mm}{\overline{m}}
\newcommand{\E}{\mathbb{E}}
\newcommand{\Prob}{\mathbb{P}}

\renewcommand{\tilde}{\widetilde}
\renewcommand{\P}{\mathbb{P}}
\newcommand{\F}{\mathcal{F}}
\renewcommand{\O}{\Omega}
\newcommand{\indc}{\mathds{1}}

\DeclareMathOperator*{\osc}{osc}

\DeclareMathOperator{\tr}{tr}

\DeclareMathOperator{\USC}{USC}
\DeclareMathOperator{\LSC}{LSC}
\DeclareMathOperator{\BUC}{BUC}

\def\XXint#1#2#3{{\setbox0=\hbox{$#1{#2#3}{\int}$}
     \vcenter{\hbox{$#2#3$}}\kern-.5\wd0}}

\begin{document}

\title[Stochastic homogenization of viscous Hamilton-Jacobi equations]{Stochastic homogenization of viscous Hamilton-Jacobi equations and applications}

\author[S. N. Armstrong]{Scott N. Armstrong}
\address{CEREMADE (UMR CNRS 7534), Universit\'e Paris-Dauphine, Paris, France}
\email{armstrong@ceremade.dauphine.fr}
\author[H. V. Tran]{Hung V. Tran}
\address{Department of Mathematics\\
The University of Chicago\\ 5734 S. University Avenue Chicago, Illinois 60637, USA}
\email{hung@math.uchicago.edu}
\date{\today}

\begin{abstract}
We present  stochastic homogenization results for viscous Hamilton-Jacobi equations using a new argument which is based only on the subadditive structure of maximal subsolutions (solutions of the ``metric problem"). This permits us to give qualitative homogenization results under very general  hypotheses: in particular, we treat non-uniformly coercive Hamiltonians which satisfy instead a weaker averaging condition. As an application, we derive a general quenched large deviations principle for diffusions in random environments and with absorbing random potentials. 
\end{abstract}

\maketitle

\section{Introduction} \label{I}

\subsection{Motivation and informal summary of results}

In this paper we consider the \emph{qualitative} stochastic homogenization of second-order, ``viscous" Hamilton-Jacobi equations. We present a new, short and self-contained argument which yields homogenization under very general and essentially optimal hypotheses. Our framework includes a class of equations for which the homogenization result has an equivalent formulation in probabilistic terms as a quenched large deviations principle (LDP) for diffusions in random environments (and/or with random obstacles), and so a corollary of our analysis is a very general such LDP for such problems which unifies many previous results on the topic.

In its time-dependent form, the viscous Hamilton-Jacobi equation we consider is
\begin{equation} \label{e.VHJ}
u^\ep_t - \ep \tr\left( A\left(\frac x\ep \right) D^2u^\ep \right) + H\!\left( Du^\ep,\frac x\ep \right) = 0 \quad \mbox{in} \ \Rd \times (0,\infty).
\end{equation}
Here $D\phi$ and $D^2\phi$ denote the gradient and Hessian of a real-valued function $\phi$, and $\tr B$ is the trace of a $d$-by-$d$ matrix $B$. The coefficients~$A$ and~$H$ are called the \emph{diffusion matrix} and the \emph{Hamiltonian}, respectively, and are assumed to be stationary-ergodic random fields. That is, they are randomly selected from the set of all such equations by an underlying probability measure which is stationary  and ergodic with respect to $\Rd$--translations. The essential structural hypotheses on the coefficients are that~$A$ takes values in the nonnegative definite matrices (and in particular may be degenerate or even vanish) and~$H$ is \emph{convex} and \emph{growing superlinearly} in its first variable. See below for some important examples of the equations which fit into our framework.

\smallskip

The presence of the~$\ep$ factor in the diffusion term of~\eqref{e.VHJ} gives the equation a critical scaling, and it turns out that it behaves like a first-order Hamilton-Jacobi equation in the limit~$\ep\to0$. Indeed, rather than providing any useful regularizing effect, the diffusion term actually makes the analysis more difficult compared to the pure first-order case by destroying localization effects (such as the finite speed of propagation). Also notice that, while we choose to write the principal part of~\eqref{e.VHJ} in nondivergence form, thanks to the scaling of the equation, our study also covers the case of equations with principal part in divergence form. Indeed, we may  rewrite an equation with principal part divergence form, at least in the case that the diffusion matrix is sufficiently smooth (on the microscopic scale) in the form of~\eqref{e.VHJ} by simply expanding out the divergence, observing that the $\ep$'s cancel, and absorbing the new first-order drift term into the Hamiltonian. 

\smallskip

The archetypical result of almost-sure, qualitative homogenization for~\eqref{e.VHJ} is that there exists a \emph{deterministic}, constant-coefficient equation
\begin{equation} \label{e.homg}
u_t + \overline H(Du)  =  0 \quad \mbox{in} \ \Rd \times (0,\infty)
\end{equation}
such that, subject to an appropriate initial condition, $u^\ep$ converges locally uniformly, as $\ep \to 0$ and with probability one, to the solution $u$ of~\eqref{e.homg}. The nonlinearity $\overline H$,  called the \emph{effective Hamiltonian}, depends on~$\P$ but is a deterministic quantity. It inherits convexity and superlinearity from the heterogeneous Hamiltonian. Its fine qualitative properties encode information regarding the behavior of solutions of the heterogeneous equation~\eqref{e.VHJ}. In the particular case corresponding to quenched large deviation principles for diffusions in random environments, $\overline H$ is, up to a constant, the Legendre-Fenchel transform of the rate function (see below for a more details).

\smallskip

The first qualitative homogenization results of this type for second-order equations, asserting that~\eqref{e.VHJ} homogenizes to a limiting equation of the form of~\eqref{e.homg}, were proved independently by~Kosygina, Rezakhanlou and Varadhan~\cite{KRV} and Lions and Souganidis~\cite{LS3}. Earlier homogenization results for first-order equations (i.e., $A\equiv 0$) in the random setting are due to Souganidis~\cite{S} and Rezakhanlou and Tarver~\cite{RT} and subsequent work can be found in~\cite{AS1,LS2,LS3}.

\smallskip

In this paper we present a new proof of homogenization which applies to a wider class of equations. The idea is to apply of the subadditive ergodic theorem to certain \emph{maximal subsolutions}, thereby obtaining a deterministic limit and a candidate for $\overline H$, and then recovering the full homogenization result by deterministic comparison arguments. The approach is simple and more or less self-contained (the reader may consult our recent paper~\cite{AT} for the necessary deterministic PDE theory) and yields a very general qualitative homogenization theorem under essentially optimal hypotheses. In addition to recovering all of the known cases, we can also treat for the first time general Hamiltonians which are not necessarily uniformly coercive. An essential characteristic of~\eqref{e.VHJ} is that $p\mapsto H(p,y)$ exhibits super linear growth in $p$, and this is typically assumed to be uniform in $x$. Here we can treat Hamiltonians satisfying an \emph{averaged} coercivity condition which is not uniform in~$x$.

\smallskip

But the most important feature of the method is that, unlike previous approaches, our proof of homogenization is \emph{quantifiable}. This will be demonstrated in the forthcoming paper of the first author and Cardaliaguet~\cite{AC}. Much recent effort has been put into obtaining quantitative stochastic homogenization results, for example, estimates for the difference $u^\ep -u$, rigorous bounds for computational methods for computing effective coefficients, and so on. For first-order Hamilton-Jacobi equations, quantitative stochastic homogenization results were recently obtained by Armstrong, Cardaliaguet and Souganidis~\cite{ACS}, who quantified the convergence proof of Armstrong and Souganidis~\cite{AS2}.  Unfortunately, the method of~\cite{AS2} is known to not be applicable in the viscous case, as the presence of the diffusion term generates significant additional difficulties. From this point of view, the results in this paper can be considered as the completion of the idea which originated in~\cite{AS2}.

\subsection{Statement of the main results}
\label{ss.hypos}

We begin by defining ``the set of all equations" by specifying some structural conditions on the coefficients. We work with parameters $\m> 1$, $n\in \N$ and $\Lambda_1\geq 1$ and $\Lambda_2 \geq 0$, which are fixed throughout the paper. 

We require the coefficients to be functions $A:\Rd\to \Sy$ (here $\Sy$ denotes the set of $d$-by-$d$ real symmetric matrices) and $H:\Rd \times \Rd \to \R$ satisfying the following conditions: first, the diffusion matrix has a Lipschitz square root. Precisely, we assume that there exists a function $\sigma :\Rd \to \R^{n\times d}$ such that
\begin{equation*} \label{}
A = \frac12\sigma^t \sigma,
\end{equation*}
where $\sigma$ is bounded and Lipschitz: for every $y,z\in \Rd$,
\begin{equation}\label{e.sigbnd}
\left| \sigma(y)\right| \leq \Lambda_2
\end{equation}
and
\begin{equation}\label{e.siglip}
\left| \sigma(y) - \sigma(z) \right| \leq \Lambda_2 |y-z|.
\end{equation}
(Here $\R^{n\times d}$ is the set of real $n$-by-$d$ matrices.)

Regarding the Hamiltonian, we assume that, for every $y\in \Rd$,
\begin{equation}\label{e.Hconvex}
p\mapsto H(p,y) \quad \mbox{is convex,}
\end{equation}
for every $R>0$, there exist constants $0<a_R\leq 1$ and $M_R \geq 1$ such that, for every $p,\hat p\in \Rd$ and $y,z\in B_R$,
\begin{equation}\label{e.Hsubq}
 a_R |p|^{\m} - M_R \leq H(p,y) \leq \Lambda_1\big( |p|^\m+1 \big),
\end{equation}
\begin{equation}\label{e.HsubqLip}
\left| H(p,y) - H(p,z) \right| \leq \big( \Lambda_1 |p|^\m + M_R \big) |y-z|, 
\end{equation}
and
\begin{equation}\label{e.HsubqDp}
\left| H(p,y) - H(q,y) \right| \leq \Lambda_1 \big( |p| + |\hat p| + 1 \big)^{\m-1} |p-\hat p|. 
\end{equation}

We define the probability space $\Omega$ to be the set of ordered pairs $(\sigma,H)$ satisfying the above conditions:
\begin{equation*} \label{}
\Omega:= \big\{ (\sigma,H) \,:\, \mbox{$\sigma$ and $H$ satisfy~\eqref{e.Hconvex},~\eqref{e.Hsubq},~\eqref{e.HsubqLip},~\eqref{e.HsubqDp},~\eqref{e.sigbnd} and~\eqref{e.siglip}}\big\}.
\end{equation*}
We may write $\Omega=\Omega(\m,n,\Lambda_1,\Lambda_2)$ if we wish to emphasize the dependence of $\Omega$ on the parameters.

We endow the set $\Omega$ with the following $\sigma$--algebra $\F$:
\begin{equation*} \label{}
\F:= \mbox{$\sigma$--algebra generated by $(\sigma,H)\mapsto \sigma(y)$ and $(\sigma,H) \mapsto H(p,y)$, with $p,y\in \Rd$. }
\end{equation*}

The \emph{random environment} is modeled by a probability measure $\P$ on $(\Omega,\F)$. The expectation with respect to $\P$ is denoted by $\E$. We assume that $\P$ is stationary and ergodic with respect to the action of $\Rd$ on $\Omega$ given by translation. To be precise, we let $\{ \tau_z \}_{z\in\Rd}$ be the group action of translation on $\Omega$ defined by
\begin{equation*} \label{}
\tau_z (\sigma,H) := \left(\tau_z\sigma, \tau_z H\right), \quad \mbox{where} \quad (\tau_z\sigma)(y) := \sigma(y+z) \quad \mbox{and} \quad (\tau_zH)(p,y) := H(p,y+z).
\end{equation*}
We extend this to $\F$ by setting, for every event $F \in \F$, 
\begin{equation*} \label{}
\tau_zF:= \big\{ \tau_z\omega\,:\, \omega\in F\big\}.
\end{equation*}
The stationary--ergodic hypothesis is that
\begin{equation} \label{e.stat}
\mbox{for all $y\in \Rd$ and $F \in \F$,} \quad \P\left[ \tau_yF\right] = \P\left[ F \right] \qquad \mbox{(stationarity)}
\end{equation}
and, for all $F\in \F$,
\begin{equation} \label{e.erg}
\bigcap_{z\in\Rd} \tau_zF=F \quad  \mbox{implies that} \quad \P\left[F\right]\in \{ 0 ,1\} \quad \mbox{(ergodicity)}.
\end{equation}
The final assumption we impose on $\P$ is a \emph{weak coercivity} condition: there exists an exponent $\alpha > d$ such that 
\begin{equation} \label{e.wkcoer}
\E \left[\left(\frac{ \Lambda_2 }{a_1}\right)^{2\alpha/(m-1)}  + \left( \frac{M_1}{a_1}\right)^{\alpha/ m}\right]  < +\infty.
\end{equation}
We emphasize that $\Lambda_2\geq 0$ is a constant but $0<a_1\leq 1$ and $M_1\geq 1$ are random variables. 

\begin{remark}
We emphasize that, in contrast to $\m$, $n$, $\Lambda_1$ and $\Lambda_2$, the positive constants $a_R$ and $M_R$ in the assumptions~\eqref{e.Hsubq} and~\eqref{e.HsubqLip} depend on $H$ itself, that is, they are random variables on~$\Omega$. To make this precise, for each $\omega=(\sigma,H)\in \Omega$, we redefine $M_R(\omega)$ to be the smallest constant not smaller than~1 for which~\eqref{e.HsubqLip} holds in $B_R$; we then redefine $a_R(\omega)$ to be the largest constant not larger than~1 for which~\eqref{e.Hsubq} holds in $B_R$. We denote 
\begin{equation*} \label{}
a_R(x,\omega):= a_R(\tau_x\omega)\quad \mbox{and} \quad M_R(x,\omega) := M_R(\tau_x\omega)
\end{equation*}
We drop the dependence on $\omega$  from the notation where possible, e.g., $a_R(x,\omega) = a_R(x)$. 
\end{remark}

We present the main homogenization result in terms of the initial-value problem
\begin{equation} \label{mainIVP}
\left\{ \begin{aligned}
& u^\ep_t - \ep \tr\left( A\left(\frac x\ep \right) D^2u^\ep \right) + H\!\left( Du^\ep,\frac x\ep \right) = 0 &  \mbox{in} & \ \Rd \times (0,\infty), \\
& u^\ep = g & \mbox{on} & \ \Rd \times \{ 0 \}.
\end{aligned} \right.
\end{equation}
Here the initial data $g$ is a given element of $\BUC(\Rd)$, the set of bounded and uniformly continuous real-valued functions on $\Rd$, and the unknown function $u^\ep$ depends on $(x,t)$ as well as $g$ and the coefficients $\omega=(\sigma,H)$. We typically write $u^\ep(x,t,g,\omega)$, or often simply $u^\ep(x,t,g)$ or $u^\ep(x,t)$. As explained in~Section~\ref{H}, under our assumptions, the problem~\eqref{mainIVP} has a unique viscosity solution (subject to an appropriate growth condition), almost surely with respect to $\P$. In fact, it is defined by the formula~\eqref{uepform2} below. We remark that all differential equations and inequalities in this paper, including the ones above, are interpreted in the \emph{viscosity} sense; see Remark~\ref{viscosity}.  

In the main result, we identify a continuous, convex $\overline H:\Rd\to\R$ and show that, as~$\ep \to 0$, the solutions~$u^\ep$ of~\eqref{mainIVP} converge,~$\P$--almost surely, to the unique solution of 
\begin{equation} \label{homgIVP}
\left\{ \begin{aligned}
& u_t + \overline H( Du ) = 0 &  \mbox{in} & \ \Rd \times (0,\infty), \\
& u = g & \mbox{on} & \ \Rd \times \{ 0 \}.
\end{aligned} \right.
\end{equation}
That the latter has a unique solution is a consequence of the properties of $\overline H$ summarized in Lemma~\ref{Hbarpropts} (see Section~\ref{H} for more details). 

We now present the statement of the main homogenization theorem. 

\begin{thm}\label{Hg}
Let $(\Omega,\F)$ be defined as above for fixed constants $m>1$ and $\Lambda_1,\Lambda_2>0$. Suppose that $\P$ is a probability measure on $(\Omega,\F)$ satisfying~\eqref{e.stat},~\eqref{e.erg} and~\eqref{e.wkcoer}. Then there exists a convex $\overline H\in C(\Rd)$ satisfying, for some constants $C,c>0$,  
\begin{equation*} \label{}
c(|p|^m-C)\leq \overline H(p) \leq C (|p|^m+1)
\end{equation*}
with the following property: with $u^\ep(x,t,g,\omega)$ defined by~\eqref{uepform2} and denoting by $u=u(x,t,g)$ the unique solution of~\eqref{homgIVP}, we have 
\begin{equation*} \label{}
\P\left[ \forall g\in \BUC(\Rd), \ \forall R>0, \   \limsup_{\ep \to 0} \sup_{(x,t)\in B_R\times[0,R)} \left| u^\ep(x,t,g) - u(x,t,g) \right| =0 \right] =1.
\end{equation*}
 \end{thm}

Let us say a few words regarding the role of the weak coercivity assumption. The first thing to notice about~\eqref{e.wkcoer} is that a particular case occurs when~$\P$ is supported on the set of~$(\sigma,H)$ for which~$H$ satisfies~\eqref{e.Hsubq} and~\eqref{e.HsubqLip} for constants $a_R>0$ and $M_R> 1$ \emph{which are independent of~$R$}. We call this a \emph{uniform coercivity} condition and it is the traditional hypothesis under which homogenization results for viscous Hamilton-Jacobi equations have been obtained. From the PDE point of view, it is important because it provides uniform Lipschitz estimates for solutions, which is a starting point for the analysis. The condition~\eqref{e.wkcoer} can then be seen as a relaxation of the uniform coercivity condition, replacing it by an averaging condition. We remark that we expect the averaging condition stated here to be optimal in terms of the range of the exponent~$\alpha$. The result should not hold if we only have~\eqref{e.wkcoer} for $\alpha=d$.

There are few homogenization results in the random setting without uniform coercivity. Armstrong and Souganidis~\cite{AS1} recently proved such a result under a less general averaging condition (essentially~\eqref{e.wkcoer} with $a_1$ bounded below). They also assumed the random environment satisfied a strong mixing condition with an algebraic mixing rate assumed to be sufficiently fast, depending on the exponent $\alpha$. Similar results stated in probabilistic terms were obtained at about the same time by Rassoul-Agha, Sepp\"al\"ainen and Yilmaz~\cite{R-ASY}. In contrast to these results, we do not require any mixing condition here, merely that the environment be stationary--ergodic. 

\smallskip

We next present a model equation which fits into our framework.

\begin{exam}
Consider the particular case of the Hamiltonian 
\begin{equation} \label{e.exam1}
H(p,y) = a(y) |p|^\m - V(y),
\end{equation}
where $m>1$, the functions $a$ and $V$ are stationary--ergodic random fields which are almost surely locally Lipschitz, $V\geq 0$ and $a$ is positive and uniformly Lipschitz on $\Rd$. This of course fits under our framework, since given such a random function $H$ (together with $\sigma$) we simply take $\P$ to be the law of $(\sigma,H)$. The weak coercivity condition is satisfied in this case provided that, for some $\alpha > d$,
\begin{equation*} \label{}
\E\left[ \left( \frac{1}{a(0)} \right)^{2\alpha/(\m-1)} + \left( \frac{\| V \|_{C^{0,1}(B_1)}}{a(0)} \right)^{\p/\m} \right] < +\infty. 
\end{equation*}
If the diffusion matrix $A$ vanishes, we only need that, for some $\alpha > d$,
\begin{equation*} \label{}
\E\left[ \left( \frac{\| V \|_{C^{0,1}(B_1)}}{a(0)} \right)^{\p/{\m}} \right] < +\infty. 
\end{equation*}
In the case that $V$ is bounded and uniformly Lipschitz, we need simply that $a^{-1} \in L^p(\O)$ for some $p > \frac{2\d}{\m-1}$; if in addition there is no diffusion ($A=0$), then we just need $p > \d/\m$. Even in these relatively simple situations, the homogenization result we obtain is completely new. In the case that $a$ is bounded below, then of course we just need that $\E \left[ \| V \|_{C^{0,1}(B_1)}^p\right] < +\infty$ for some $p > d/\m$, which is better than the condition $p \geq d$ assumed in~\cite{AS1}.
\end{exam}

\begin{remark}
It is customary in the homogenization literature to hide the specifics of the probability space $\Omega$ by introducing the ``dummy variable"~$\omega$ and expressing $\sigma$ and $H$ as maps $\sigma:\Rd\times\Omega\to \Sy$ and $H:\Rd\times \Rd\times \Omega\to \R$ by identifying $\sigma(\cdot,\omega)$ and $H(\cdot,\cdot,\omega)$ with $\tilde \sigma$ and $\tilde H$, respectively, where $\omega=(\tilde \sigma,\tilde H)$. Viewed this way, the functions $A$ and $H$ are \emph{stationary} with respect to the translation group action $\{ \tau_z\}_{z\in\Rd}$ in the sense that, for every $p,y,z\in \Rd$ and $\omega\in\Omega$,
\begin{equation*}\label{}
\sigma(y,\tau_z\omega) =\sigma(y+z,\omega) \qquad \mbox{and} \qquad H(p,y,\tau_z\omega) = H(p,y+z,\omega).
\end{equation*}
While this is evidently equivalent to the formulation here, we feel that writing $\omega$ everywhere is both unsightly and unnecessary and so we avoid it wherever possible. The meaning of expressions like $\P[\,\cdots]$ and $\E[\,\cdots]$ are always quite clear from the context. Meanwhile, measurability issues are taken care of quite cleanly by the definition of $\F$ and become, in our opinion, more rather than less confusing if we display explicit dependence on $\omega$. 
\end{remark}

\subsection{A quenched LDP for diffusions in random environments}

In order to state the main probabilistic application of Theorem~\ref{Hg}, we require some additional notation. We begin first with another example of a Hamilton-Jacobi equation with random coefficients which is contained in the framework of Theorem~\ref{Hg}.

\begin{exam}
\label{exam.ldp}
With $\sigma:\R^d \to \R^{d\times d}$ as described in the hypotheses (with $n=d$) and given a random vector field $b$ and potential $V\geq 0$, we define the Hamiltonian
\begin{equation} \label{e.HforLDP}
H(p,y) = \frac12\left| \sigma p\right|^2+ b(y) \cdot p - V(y) = p \cdot Ap + b(y) \cdot p - V(y),
\end{equation}
where as usual $A=\frac12\sigma^t\sigma$. The weak coercivity condition is satisfied provided there exists $\alpha > d$ such that
\begin{equation} \label{wcLDP}
\E \left[ \left( \frac{1}{\lambda_1(A(0))} \right)^{2\alpha}    + \left(  \frac{\| V \|_{C^{0,1}(B_1)}}{\lambda_1(A(0))} \right)^{\alpha/2}  \right] < +\infty,
\end{equation}
where $\lambda_1(A) = \frac12\min_{|z|=1} |\sigma z|^2$ is the smallest eigenvalue of $A$. In this random variable is bounded below we say that $A$ is \emph{uniformly elliptic}, and in this case we need only that the potential $V$ have a finite $q$th moment for some $q>d/2$. 
\end{exam}

Throughout the rest of this subsection, we take $\sigma$, $A$, $b$ and $V$ to be as in Example~\ref{exam.ldp}. In this situation, we may identify the probability space $\Omega$ with ordered triples $(\sigma,b,V)$. 

\smallskip

We denote by $X_\cdot=\{X_t\}_{t\geq 0}$ the canonical process on $C\!\left(\R_+,\Rd\right)$. Recall that the martingale problem corresponding to $\sigma$ and $b$ has a unique solution (c.f.~\cite{SV}). This means that, for each $x\in \Rd$ and $\omega=(\sigma,b,V)\in \O$, there exists a unique probability measure $P_{x,\omega}$ on $C\!\left(\R_+,\Rd\right)$ such that, under $P_{x,\omega}$, the canonical process $X=\{ X_t\}_{t\geq 0}$ satisfies the stochastic differential equation
\begin{equation*} \label{}
\left\{ \begin{aligned}
& dX_t = \sigma\left(X_t,\omega\right) dB_t + b(X_t,\omega)dt, \\
& P_{x,\omega}\left[ X_0 =x \right] = 1,
\end{aligned} \right.
\end{equation*}
where $\{ B_t \}_{t\geq 0}$ is a $d$-dimensional Brownian motion with respect to $P_{x,\omega}$.

The main object of interest is the \emph{quenched path measure of the diffusion $X_\cdot$ in the random potential $V(\cdot,\omega)$}, which is defined, for each $x\in \Rd$, $\omega\in \O$ and $t>0$, by:
\begin{equation*} \label{}
Q_{t,x,\omega} (dv) := \frac{1}{S(t,x,\omega)} \exp\left( - \int_0^t V(X_s,\omega)\, ds \right) P_{x,\omega}(dv), 
\end{equation*}
where the normalizing factor $S(t,x,\omega)$, called the \emph{quenched partition function}, is given by
\begin{equation} \label{e.Sdef}
 S(t,x,\omega):= E_{x,\omega} \left[  \exp\left( - \int_0^t V(X_s,\omega)\, ds \right) \right].
\end{equation}
Note that $Q_{t,x,\omega}$ is a probability measure on the path space $C(\R_+;\Rd)$.

\smallskip

The physical interpretation of the quenched path measures is that $Q_{t,x,\omega}$ describes the behavior of the diffusion~$X$ in an ``absorbing" potential (in this interpretation, the half-life of a particle at position~$x$ is~$\log 2 / V(x,\omega)$) conditioned on the (exponentially unlikely event) that $X$ is not absorbed up to time $t$; the probability that the particle lives until time~$t$ is precisely~$S_{t,x,\omega}$. We note that the case that~$V\equiv 0$ is also of interest, in which case~$Q_{t,x,\omega} = P_{x,\omega}$ and our results below describe the quenched large deviations of $P_{x,\omega}$, that is, of the diffusion in the random medium with no absorption. We also remark that we may allow for $V$ taking negative values, provided that $V$ is uniformly bounded below; in the particle interpretation, negative values of $V$ correspond to the creation of particles. 

A central task in the study of diffusions in random environments is to obtain statistical information about the typical sample paths under $Q_{t,x,\omega}$. Here we are interested in information regarding the large deviations of $Q_{t,x,\omega}$ in the asymptotic limit $t\to \infty$.

\begin{cor}\label{c.ldp}
Let~$\P$ be a probability measure on $\Omega$ (which is identified with ordered tripes $(\sigma,b,V)$ as explained above) satisfying~\eqref{e.stat},~\eqref{e.erg} and~\eqref{wcLDP}. Let $\barH$ be as in the statement of Theorem~\ref{Hg} corresponding to the Hamiltonian~$H$ given in~\eqref{e.HforLDP}, and let $\barL$ be the Legendre-Fenchel transform of $\barH$, defined for $z \in \Rd$ by
\begin{equation*} \label{}
\barL(z):= \sup_{p\in\Rd} \left( p\cdot z - \barH(p) \right). 
\end{equation*}  
Then there exists $\O_0 \in \F$ with $\P[\Omega_0]=1$, such that, for every $\omega\in \O_0$, we have the following:
\begin{enumerate}
\item[(i)] For every closed set $K \subseteq \Rd$ and $x\in \Rd$,
\begin{equation} \label{ldp.uppbnd}
\liminf_{t\to \infty} -\frac1t \log Q_{t,tx,\omega} \left[ X_t\in tK \right] \geq \inf_{y\in K} \barL(x-y) + \barH(0).
\end{equation}

\item[(ii)] For every open set $U \subseteq \Rd$,
\begin{equation} \label{ldp.lwrbnd}
\limsup_{t\to \infty} -\frac1t \log Q_{t,tx,\omega} \left[ X_t\in tU \right] \leq \inf_{y\in U} \barL(x-y)+ \barH(0).
\end{equation}
\end{enumerate}
\end{cor}

The proof that Theorem~\ref{Hg} implies Corollary~\ref{c.ldp} is presented in Section~\ref{QLD}.

Sznitman~\cite{Sz1} was the first to prove a quenched large deviations result like this in dimensions larger than one. Precisely, he proved Corollary~\ref{c.ldp} in the special case that $\sigma=I_\d$ is the identity matrix, $b(y,\omega) = b_0\in \Rd$ is a constant vector, and the potential $V$ is a \emph{Poissonian potential}; i.e., 
\begin{equation*} \label{}
V(y,\omega) = \int_{\Rd} W(y-z) \, d\rho(z)
\end{equation*}
where $W\in C^{\infty}_c(\Rd)$ and the locally finite measure $\rho$ has a Poissonian law (see~\cite[Theorem~4.7]{Szb}). Note that such a potential has a finite range of dependence and bounded moments. 

The first phase of the strategy followed in this paper to homogenize the Hamilton-Jacobi equation is analogous to the probabilistic approach Sznitman used to obtain the large deviation principle. His proof relied on an application of the subadditive ergodic theorem to certain quantities, essentially equivalent to  our $m_\mu$'s, to obtain deterministic limits which he called the \emph{Lyapunov exponents}, which are precisely our $\mm_\mu$'s. See also the discussion preceding Proposition~\ref{shape}.

So that the reader may see that the rate function in Corollary~\ref{c.ldp} agrees with the one in~\cite{Szb}, we note that $\min_{\Rd} \barH=\barH(0)=0$ in Sznitman's case and that the effective Lagrangian~$\barL$ may be expressed in terms of the $\mm_\mu$'s as follows:
\begin{align*} \label{}
\barL(z) & = \sup_{z\in \Rd} \left( p\cdot z - \barH(p) \right) & \mbox{(definition of $\barL$)}\\
& = \sup_{\mu > 0} \sup \left\{  p\cdot z - \barH(p) \, : \, \barH(p) \leq \mu \right\} & \mbox{(by $0= \min \barH$)} \\
& = \sup_{\mu > 0} \sup \left\{  p\cdot z - \mu \, : \, \barH(p) \leq \mu \right\}  \\ & = \sup_{\mu > 0} \left( \mm_\mu(z) - \mu \right). & \mbox{(by~\eqref{mm} below)}
\end{align*}
In the absorption-free case $V\equiv 0$, Zerner~\cite{Z} proved a result similar to Corollary~\ref{c.ldp} for random walks on the lattice $\Zd$ with i.i.d.~transition probabilities at each lattice point. He required (loosely translated into our notation) that $A$ be ``almost" uniformly elliptic: 
\begin{equation} \label{e.Zcond}
\E \left[ -\log \lambda_1(A(0,\omega))^\d \right]< \infty.
\end{equation}
This condition is much weaker than our~\eqref{wcLDP} but is compensated for by the much stronger independence assumption on the random environment.

The subject of large deviations of random walks in random environments continues to receive much attention, and the works of Sznitman and Zerner have been subsequently extended to more general settings and properties of the rate function have been studied in more depth; in particular, we refer to~Varadhan~\cite{V} and Rassoul-Agha~\cite{R-A}. See also the more recent work of Yilmaz~\cite{Y} who proves a discrete version of Corollary~\ref{c.ldp} with no absorption, $V=0$, in a quite general stationary-ergodic framework like ours with a slight strengthening of~\eqref{e.Zcond}. Finally, a large deviation result for random walks in the case of absorption, $V\not\equiv 0$, was proved recently by Rassoul-Agha, Sepp\"al\"ainen and Yilmaz~\cite{R-ASY} under the assumptions that the random environment is uniform ellipticity and strongly mixing. Admitting the proof of Corollary~\ref{c.ldp} from Theorem~\ref{Hg}, the results of~\cite{R-ASY} may be compared to those of Armstrong and Souganidis~\cite{AS1}.

Finally, we mention that the connection between large deviations and viscosity solutions of Hamilton-Jacobi equations was observed by Evans and Ishii~\cite{EI}, who studied large deviations of the occupation times of small random perturbations of ODEs. 

\subsection{Disclaimer on viscosity solutions}
\label{viscosity}
Throughout the paper, all differential equalities and inequalities are understood in the viscosity sense. For a general introduction to viscosity solutions, we refer to~\cite{CIL}. Many of the fundamental PDE results we need here are proved in~\cite{AT}, which cite many times below. Recall that the natural function space for viscosity subsolutions is the set~$\USC(X)$ of upper semicontinuous functions on domain~$X$, and for supersolutions it is the set~$\LSC(X)$ of lower semicontinuous functions on~$X$.

\subsection{Outline of the paper}
In the next section we introduce the maximal subsolutions and homogenize them using the subadditive ergodic theorem. In Section~\ref{eff}, we construct the effective Hamiltonian and study some of its basic properties. In Section~\ref{MAC} we give the proof of an intermediate homogenization result and finally prove Theorem~\ref{Hg} in Section~\ref{H}. The quenched large deviations principle is shown in Section~\ref{QLD} to be a consequence of the homogenization result.

\section{The shape theorem: homogenization of the maximal subsolutions} \label{barH}

In this section we homogenize the \emph{maximal subsolutions} of the inequality
\begin{equation} \label{e.maxsube}
-\tr \left( A(y) D^2w \right) + H( Dw,y) \leq \mu \quad \mbox{in} \ \Rd.
\end{equation}
These are defined, for each $\mu \in \R$ and $y,z\in \Rd$, by 
\begin{equation}
m_\mu(y,z):=\sup \Bigg\{ w(y) - \sup_{\overline B_1(z)} w \,:\, w \in \USC(\Rd) \ \ \mbox{satisfies~\eqref{e.maxsube}} \Bigg\}.
\end{equation}
If the admissible class in the supremum above is empty, then we take $m_\mu(y,z) \equiv -\infty$. We denote, for every $\omega=(\sigma,H) \in \Omega$, the critical parameter $h(\omega)$ for which $m_\mu$ is finite by
\begin{equation}\label{eq.h}
h:= \inf \left\{ \mu \,:\, \mbox{there exists} \ w \in \USC(\Rd) \ \ \mbox{satisfying~\eqref{e.maxsube}} \right\}.
\end{equation}
According to~\eqref{e.Hsubq}, we have $h(\omega) \leq \Lambda_1$. It is sometimes convenient to work with the quantity
\begin{equation}\label{tildemmu}
\tilde m_\mu(y,z):= \sup_{B_1(y)} m_\mu(\cdot,z).
\end{equation}

Some deterministic properties of the maximal subsolutions are summarized in the following proposition, which is proved in~\cite{AT}. See Proposition~3.1 and~Section 5 of that paper. The estimate~\eqref{e.oscloc.p} below is particularly important in our analysis, and comes from the explicit Lipschitz estimates proved in~\cite[Proposition 3.1]{AT}.

\begin{prop}[\cite{AT}]
Fix $\omega=(\sigma,H)\in \Omega$ and $\mu \geq h(\omega)$. Then, for every $z\in \Rd$, the function $m_\mu(\cdot,z)$ belongs to $C^{0,1}_{\mathrm{loc}}(\Rd\setminus \overline B_1(z)) \cap \USC(\Rd)$ and satisfies
\begin{equation}\label{m-11}
-\tr(A(y)D^2 m_\mu)+H(Dm_\mu,y)\leq \mu \quad \mbox{in}\ \Rd 
\end{equation}
as well as
\begin{equation}\label{m-12}
-\tr(A(y)D^2 m_\mu)+H(Dm_\mu,y) = \mu \quad \mbox{in}\ \Rd \setminus \overline B_1(z).
\end{equation}
There exists a constant $C>0$, depending only on $d$ and $\m$, such that, for every $y,z\in \Rd$, 
\begin{equation} \label{e.oscloc.p}
\osc_{B_1(y)} m_\mu(\cdot,z) \leq C \left[ \left( \frac{(1+\Lambda_1)^{1/2}\|\sigma \|_{C^{0,1}(B_2(y))} }{a_2(y)} \right)^{2/(\m-1)} + \left( \frac{M_2(y)+\mu}{a_2(y)} \right)^{1/\m}  \right].
\end{equation}
For every $\lambda \in [0,1]$, $\mu,\nu \geq h(\omega)$ and $y,z\in \Rd$,
\begin{equation} \label{e.concv}
m_{\lambda \mu + (1-\lambda)\nu}(y,z) \geq \lambda m_{ \mu}(y,z) + (1-\lambda)m_\nu(y,z).
\end{equation}
Finally, for every $x,y,z\in\Rd$, we have
\begin{equation} \label{e.subadd}
\widetilde m_\mu(y,z) \leq \widetilde m_\mu(y,x) + \widetilde m_\mu(x,z).
\end{equation}
\end{prop}

We define $K_\mu(y)$ to be the random variable on the right side of~\eqref{e.oscloc.p}, that is,
\begin{equation*} \label{}
K_\mu(y) := C \left[ \left( \frac{(1+\Lambda_1)^{1/2}\|\sigma \|_{C^{0,1}(B_2(y))} }{a_2(y)} \right)^{2/(\m-1)} + \left( \frac{M_2(y)+\mu}{a_2(y)} \right)^{1/\m}  \right]
\end{equation*}
so that we can write the bound~\eqref{e.oscloc.p} as
 \begin{equation} \label{e.oscloc}
\osc_{B_1(y)} m_\mu(\cdot,z) \leq K_\mu(y). 
\end{equation}
We also denote $K_\mu=K_\mu(0)$. The primary use of the weak coercivity hypothesis~\eqref{e.wkcoer} is that it implies that the~$\alpha$th moment of~$K_\mu$, which we denote by $\overline K_\mu^\alpha$, is finite for some~$\alpha>d$:
\begin{equation} \label{smack}
\overline K_\mu := \E \left[ K_\mu^\alpha \right]^{1/\alpha}  < +\infty.
\end{equation}
Note that we have used~\eqref{e.wkcoer} with $a_2$ and $M_2$ replacing $a_1$ and $M_1$, respectively, which is seen to be equivalent to~\eqref{e.wkcoer} by an easy covering argument.

As far as the dependence of $\overline K_\mu$ on $\mu$, we use $M_2\geq 1$ to check that 
\begin{equation} \label{e.Kmugrmu}
\overline K_\mu \leq \overline K_0 \left( 1 + \mu^{1/m} \right).
\end{equation}

We next use ergodicity to show that the random variable $h$ defined in~\eqref{eq.h} is, up to an event of probability zero, a deterministic constant.

\begin{lem}\label{l.detcon}
Assume that $\P$ is a probability measure on $(\Omega,\F)$ satisfying~\eqref{e.stat} and~\eqref{e.erg}. Then there exists a constant $\overline H_*\in \R$, depending on $\P$, such that 
\begin{equation} \label{e.detcon}
\P \Big[ \overline H_* = \inf\left\{ \mu \in \R\,:\,  \mbox{there exists} \ \ w \in \USC(\Rd) \ \mbox{satisfying~\eqref{e.maxsube}}  \right\}  \Big]  =1. 
\end{equation}
\end{lem}
\begin{proof}
Let us see that~$h$ defined in~\eqref{eq.h} is finite. We have already seen that $h \leq \Lambda_1$ by~\eqref{e.Hsubq}. To argue that $h(\omega)>-\infty$ for every $\omega=(\sigma,H)\in \Omega$, we use the test function
\begin{equation*} \label{}
\phi(y):= k \left( 1 - |y|^2 \right)^{-1/(\m-1)}.
\end{equation*}
If $k>1$ and $C>1$ are sufficiently large, depending only on $\Lambda_2$ and the constants $a_1$, $M_1$ in~\eqref{e.Hsubq} for $H$, then $\phi$ is a smooth solution of
\begin{equation*} \label{}
-\tr \left( A(y) D^2\phi \right) + H( D\phi,y) > -C \quad \mbox{in} \ B_1.
\end{equation*}
Now consider an arbitrary element $w \in \USC(\Rd)$. Since $\phi(y) \to +\infty$ as $y\to \partial B_1$, there exists $x_0\in B_1$ such that $w-\phi$ has a local maximum at $x_0$. In view of the differential inequality for $\phi$, we obtain that $w$ cannot be a subsolution of~\eqref{e.Hsubq} for any $\mu \geq -C$.

It is immediate from its definition that $h$ is invariant under the translation group action $\{ \tau_y\}_{y\in \Rd}$. By the ergodicity assumption, this implies that $\P$ assigns each of the events $\{ h>\lambda \}$ and $\{ h < \lambda \}$, for every $\lambda \in \R$, probability either zero or one. This implies that $h$ is $\P$--almost surely a constant. Taking this constant to be $\overline H_*$ yields the lemma.
\end{proof}

Our main interest lies in the asymptotic behavior of $m_\mu(y,z)$ for $|y-z| \simeq |z| \gg 1$. In the next lemma we use Morrey's inequality together with the local oscillation bound~\eqref{e.oscloc} and the ergodic theorem to prove the large scale oscillation bound $\osc_{B_{R}(Ry)}m_\mu(\cdot,z)\lesssim R$, uniformly in $z\in \Rd$ for $R\gg1$. Recall that Morrey's inequality (c.f.~\cite[Section 5.6.2]{Ebook}) states  that, for any $R>0$, $u\in C^1(B_R)$ and $\beta>\d$, there exists $C(\beta,d)>1$ such that
 \begin{equation} \label{e.morrey}
\osc_{B_R} u \leq C R \left( \fint_{B_R} \left|Du(x)\right|^\beta\,dx \right)^{1/\beta}.
\end{equation}
So we can control the oscillation of a function in terms of ``averaged pointwise oscillation bounds." Thus it is natural to attempt to control the large scale oscillation of $m_\mu(\cdot,z)$ in terms of the average of a power of its local oscillation, with the hope of using~\eqref{e.oscloc},~\eqref{smack} and the ergodic theorem to control the latter.

\begin{lem}\label{l.oscbnde}
Assume that $\P$ is a probability measure on $(\Omega,\F)$ satisfying~\eqref{e.stat},~\eqref{e.erg} and~\eqref{e.wkcoer}. Then there exists $C>0$, depending only on $d$ and $\alpha$, such that
\begin{equation} \label{e.oscbnde}
\P \left[ \forall \mu \geq \overline H_*, \ \forall x\in \Rd, \ \limsup_{R\to \infty} \sup_{z\in\Rd} \frac1R\osc_{B_R(Rx)} m_\mu(\cdot,z) \leq C\overline K_\mu \right] =1.
\end{equation}
\end{lem}
\begin{proof}
It is convenient to mollify the functions in order to put the local oscillation bounds into a pointwise form suitable for the application of Morrey's inequality. We first observe that, owing to Lemma~\ref{l.detcon}, we may assume that $m_\mu$ is finite for all $\mu \geq \overline H_*$ by removing an event of zero probability. 

We now fix $\mu\geq \overline H_*$, $z\in \Rd$ and take a nonnegative  $\eta \in C^\infty_c(\Rd)$ with support in $B_{1/2}$ and unit mass, $\int_{\Rd} \eta(y)\,dy=1$, and set
\begin{equation}\label{e.hatmmu}
\widehat m_\mu(y):= \int_{\Rd} \eta(y-x) m_\mu(x,z)\, dx.
\end{equation}
Then $\widehat m_\mu$ is smooth and using~\eqref{e.oscloc.p} we have, for every $y\in \Rd$,
\begin{align}\label{e.hatmubnd}
\left| \widehat m_\mu(y) - m_\mu(y,z) \right| & \leq \int_{\Rd} \eta(y-x) \left| m_\mu(x,z) - m_\mu(y,z) \right| \, dx \\ & \leq \osc_{B_{1/2}(y)} m_\mu(\cdot,z) \leq \inf_{B_{1/2}(y)} K_\mu(\cdot)\nonumber
\end{align}
and
\begin{equation}\label{e.Dhatmubnd}
\left| D\widehat m_\mu(y) \right| = \left|  \int_{\Rd} D\eta(y-x) \left( m_\mu(x,z) - m_\mu(y,z) \right)\, dx\right| \leq C K_\mu(y).
\end{equation}
Applying~\eqref{e.morrey} and then using~\eqref{e.Dhatmubnd}, we deduce the existence of $C(d,\alpha)>1$ such that, for every $x\in \Rd$,
\begin{align}\label{e.hatsosc}
\osc_{B_R(x)} \widehat m_\mu & \leq C R \left( \fint_{B_R(x)} \left| D \widehat m_\mu (y)\right|^\p\,dy \right)^{1/\p}  \leq C R \left( \fint_{B_R(x)} K_\mu^\p(y) \, dy \right)^{1/\p}.
\end{align}
Next, we return to~\eqref{e.hatmubnd} and observe that
\begin{align*}
\sup_{y\in B_R(x)} \left| \widehat m_\mu(y) - m_\mu(y,z) \right| & \leq \sup_{y\in B_R(x)} \ \inf_{x\in B_{1/2}(y)} K_\mu(x) \\
& \leq \left( \sup_{y\in B_R(x)} \ \fint_{B_{1/2}(y)} K_\mu^\alpha (x)\,dx \right)^{1/\alpha} \\
& \leq C\left( \int_{B_{R+1}(x)} K_\mu^\alpha(x)\, dx\right)^{1/\alpha} \\
& \leq C (R+1)^{d/\alpha} \left( \fint_{B_{R+1}(x)} K_\mu^\alpha(x)\,dx\right)^{1/\alpha}.
\end{align*}
Making note of the fact that $d/\alpha < 1$ and combining the above inequality with~\eqref{e.hatsosc}, we deduce that, for every $R>1$ and $x,z\in \Rd$,
\begin{equation} \label{e.osclocR}
\frac1R \osc_{B_R(x)} m_\mu(\cdot,z) \leq C\left( \fint_{B_{R+1}(x)} K_\mu^\alpha(y)\,dy \right)^{1/\alpha}.
\end{equation}
According to the ergodic theorem (c.f.~Becker~\cite{Be}),
\begin{equation*} \label{}
\P\left[ \lim_{R \to \infty} \left( \fint_{B_{R+1}(Rx)} K_\mu^\alpha(y)\,dy \right)^{1/\alpha} = \E \left[ K_\mu^\alpha \right]^{1/\alpha}  \right] = 1. 
\end{equation*}
In view of the definition of $\overline K_\mu$, the last two lines yield that, for every $\mu \geq \overline H_*$,
\begin{equation*} \label{}
\P \left[ \forall x\in \Rd,  \  \ \limsup_{R\to \infty} \sup_{z\in\Rd} \frac1R \osc_{B_R(Rx)} m_\mu(\cdot,z) \leq C \overline K_\mu \right] = 1. 
\end{equation*}
Using the monotonicity of $\mu \to m_\mu$ and the continuity of $\mu\mapsto\overline K_\mu$ and intersecting the events corresponding to all rational $\mu$ and $\mu=\overline H_*$, we obtain~\eqref{e.oscbnde}. 
\end{proof}

The following lemma is an abstract tool which allows us to obtain uniform convergence, with respect to the translation group $\{\tau_y\}_{y\in\Rd}$, for sequences of random variables which converge almost surely and satisfy appropriate oscillation bounds. The argument follows an idea of Varadhan, using a combination of Egoroff's theorem and the Ergodic theorem. 

\begin{lem} \label{eggg}
Assume $\P$ is a probability measure on $(\Omega,\F)$ satisfying~\eqref{e.stat} and~\eqref{e.erg}. Suppose that $\{ X_t\}_{t > 0}$ is a family of $\F$--measurable random variables on $\Omega$ such that 
\begin{equation*} \label{}
\P \left[ \limsup_{t \to \infty} X_t \leq 0 \right]  = 1.
\end{equation*}
Denote $X_t(y,\omega):= X_t (\tau_y\omega)$ and suppose that
\begin{equation*} \label{}
\P \left[ \forall z\in \Rd, \ \limsup_{r\to 0} \limsup_{t \to \infty} \osc_{y\in B_{tr}(tz)} X_t(y,\cdot) = 0 \right] = 1. 
\end{equation*}
Then 
\begin{equation*} \label{}
\P \left[ \forall R>0, \ \limsup_{t \to \infty} \sup_{y\in B_{tR} } X_t(y,\cdot) \leq 0 \right] = 1. 
\end{equation*}
\end{lem}
\begin{proof}
We first notice that, after a routine covering argument, the second hypothesis can be rewritten in a slightly stronger way as
\begin{equation} \label{secbold}
\P \left[ \forall R>0,  \ \limsup_{r\to 0} \limsup_{t \to \infty} \sup_{z\in B_R} \osc_{y\in B_{tr}(tz)} X_t(y,\cdot) = 0 \right] = 1. 
\end{equation}
By the first hypothesis, for each $\ep > 0$, there exists $T_\ep>0$ sufficiently large that
\begin{equation}\label{}
\P\left[ \sup_{t\geq T_\ep} X_t(0,\cdot) \leq \ep \right] \geq 1 - \frac12 \ep^{d}.
\end{equation}
Denote this event by $D_\ep := \left\{ \omega\in\Omega\,:\, \sup_{t\geq T_\ep} X_t(0,\omega) \leq \ep \right\}$. According to the multiparameter ergodic theorem (c.f.~\cite{Be}), for each $\ep > 0$, there exists an event $\tilde\Omega_\ep \in \mathcal F$ with $\P [\tilde \Omega_\ep ] = 1$, such that, for every $\omega\in \tilde \Omega_\ep$, 
\begin{equation}\label{}
\lim_{r\to \infty} \fint_{B_r} \indc_{D_\ep}(\tau_x\omega) \, dx = \Prob\left[ D_\ep \right] \geq 1- \frac12\ep^d. 
\end{equation}
Here $\indc_{E}$ denotes the indicator function of an event $E\in\F$.
It follows that, for each $\omega \in \tilde\Omega_\ep$, there exists $r_\ep  > 0$ sufficiently large (and depending on $\omega$ in addition to $\ep$) that
\begin{equation}\label{goodpots}
\inf_{r\geq r_\ep} \fint_{B_r} \indc_{D_\ep}(\tau_x\omega) \, dx  > 1- \ep^d. 
\end{equation}
Notice that~\eqref{goodpots} implies that, for $r\geq r_\ep(\omega)$,
\begin{equation}\label{greatpots}
\left| \left\{ y\in B_r \,:\, \tau_y\omega \in D_\ep  \right\} \right| > \big( 1-\ep^d\big) |B_r|.
\end{equation}
In particular, if $r\geq r_\ep(\omega)$ then no ball of radius $r \ep$ is contained in $\{ x \in B_r \, :\, \tau_x\omega \not\in D_\ep \}$.

Let $\tilde \Omega$ be the intersection of $\tilde \Omega_\ep$ over all positive $\ep \in \Q$. Fix $R,\ep > 0$ with $\ep \in \Q$, $\omega \in \tilde\Omega$ such that $\omega$ also belongs to the event inside the probability in~\eqref{secbold}, $t \geq R^{-1}\max\{ r_\ep(\omega) , T_\ep\}$ and $y\in B_{tR}$. Then there exists $z \in B_{R}$ such that $\tau_{tz} \omega \in D_\ep$ and $|y-tz| \leq tR\ep$. Note that $\tau_{tz} \omega \in D_\ep$ is equivalent to $X_t(tz,\omega) \leq \ep$. We deduce that
\begin{equation*} \label{}
X_t (y,\omega) \leq X_t(tz,\omega) + \osc_{x\in B_{tR\ep}(tz)} X_t(x,\omega) \leq \ep + \sup_{z'\in B_R} \osc_{x\in B_{tR\ep}(tz')} X_t(x,\omega).
\end{equation*}
This holds for all $y\in B_{tR}$, hence
\begin{equation*} \label{}
\sup_{y\in B_{tR} } X_t (y,\omega) \leq \ep + \sup_{z'\in B_R} \osc_{x\in B_{tR\ep}(tz')} X_t(x,\omega).
\end{equation*}
We have shown that, for all $\ep \in \Q$ such that $\ep > 0$, we have
\begin{equation*} \label{}
\limsup_{t \to \infty} \sup_{y\in B_{tR} } X_t (y,\omega) \leq \ep + \limsup_{t\to \infty}  \sup_{z'\in B_R} \osc_{x\in B_{tR\ep}(tz')} X_t(x,\omega).
\end{equation*}
Sending $\ep \to 0$, using that $\omega$ belongs to the event inside the probability in~\eqref{secbold}, we obtain 
\begin{equation*} \label{}
\limsup_{t \to \infty} \sup_{y\in B_{tR} } X_t (y,\omega) \leq 0. 
\end{equation*}
This conclusion applies for every $R>0$ and $\omega$ belonging to the intersection of $\tilde \Omega$ and the event in~\eqref{secbold}, which has probability one. 
\end{proof}

We next employ the subadditive ergodic theorem (c.f.~\cite{AK}) and the subadditivity of $m_\mu$ to get the following result, which asserts that, for large $t>0$, we have $m_\mu(ty,tz) \approx t \overline m_\mu(y-z) + o(t)$ for some deterministic function $\overline m_\mu$. The key ingredients in the proof are subadditivity~\eqref{e.subadd} and the local oscillation estimate~\eqref{e.oscbnde}. 

The terminology ``shape theorem" originated in first-passage percolation and ``shape" refers to the sublevel sets of $m_\mu$. In particular, the result here generalizes~\cite[Theorem ~5.2.5]{Szb} and also covers the case that $A\equiv 0$ and the Hamiltonian has the specific form $H(p,x) = a(x) |p|$ where $a>0$ is an appropriate random field, which is a continuum analogue of the first passage percolation model.

\begin{prop}[The Shape Theorem]\label{shape}
Assume $\P$ is a probability measure on $(\Omega,\F)$ satisfying~\eqref{e.stat},~\eqref{e.erg} and~\eqref{e.wkcoer}. Then there exists a family $\left\{ \overline m_\mu \,:\, \mu \geq \overline H_* \right\}\subseteq C(\Rd)$ of convex, positively homogeneous functions, such that
\begin{equation} \label{mphg}
\P\left[ \forall \mu \geq \overline H_*, \ \forall R>0, \ \limsup_{t\to \infty} \sup_{y,z\in B_R}\left| \frac{m_\mu(ty,tz)}t -  \mm_\mu(y-z) \right| = 0 \right] = 1.
\end{equation}
\end{prop}

\begin{proof}
We break the argument into five steps. In the first step, we construct $\mm_\mu$ using the subadditive ergodic theorem and, in Step 2, derive some of its basic properties. In Step 3, we prove~\eqref{mphg} for $z=0$ and in the fourth step we remove this restriction. For the first four steps we fix $\mu \geq \overline H_*$. The universal quantifier over $\mu \geq \overline H_*$ will be smuggled inside the probability in the final step.

Before commencing with the argument, we make a reduction. With $\tilde m_\mu$ defined as in~\eqref{tildemmu}, we observe that
\begin{equation*} \label{}
0 \leq \tilde m_\mu (y,z) - m_\mu(y,z) = \sup_{\xi\in B_1(y)} \left( m_\mu(\xi,z) - m_\mu(y,z) \right) \leq \osc_{B_1(y)} m_\mu(\cdot,z). 
\end{equation*}
Using this together with Lemma~\ref{l.oscbnde}, we find that
\begin{align*}
\lefteqn{\P\left[ \forall \mu \geq \overline H_*, \ \forall R>0, \ \limsup_{t\to \infty} \sup_{y,z\in B_R} \frac1t \left| m_\mu(ty,tz) - \tilde m_\mu(ty,tz) \right| = 0 \right]} \qquad & \\
& \geq \P\left[ \forall \mu \geq \overline H_*, \ \forall R>0, \ \limsup_{t\to \infty} \sup_{y,z\in B_R} \frac1t \osc_{B_1(ty)} m_\mu(\cdot,tz) = 0 \right] \\
& \geq \P\left[ \forall \mu \geq \overline H_*, \ \forall R,\delta >0,  \ \limsup_{t\to \infty} \sup_{z\in\Rd} \sup_{y\in B_R} \frac1t \osc_{B_{t\delta}(ty)} m_\mu(\cdot,z) \leq C\overline K_\mu\delta \right] = 1.
\end{align*}
Therefore, it suffices to prove the proposition with $\tilde m_\mu$ in place of $m_\mu$.

\smallskip

{\it Step 1.} We apply the subadditive ergodic theorem to construct $\overline m_\mu$. Note that it is immediate from the definitions that both $m_\mu$ and $\tilde m_\mu$ are jointly stationarity in $(y,z)$. Precisely, we mean that, using the notation $m_\mu(y,z,\omega)$ and $\tilde m_\mu(y,z,\omega)$ to denote dependence on $\omega\in\Omega$, then with respect to the translation group action $\{ \tau_x\}_{x\in\Rd}$, we have 
\begin{equation*} \label{}
m_\mu(y,z,\tau_x\omega) = m_\mu(y+x,z+x,\omega) \quad \mbox{and} \quad \tilde m_\mu(y,z,\tau_x\omega) = \tilde m_\mu(y+x,z+x,\omega).
\end{equation*}
Note that $\tilde m_\mu$ subadditive by~\eqref{e.subadd} and $\P$--integrable on $\Omega$ since~\eqref{e.osclocR} implies
\begin{multline} \label{e.mtildlip}
\E\left[ \tilde m_\mu(y,z)\right]  \leq \E \left[ \sup_{B_{|y-z|+1}(z)} m_\mu (\cdot,z)\right]  \\ \leq C\left(|y-z|+1\right) \E\left[  \left(  \fint_{B_{|y-z|+1}} K_\mu^\p(x) \, dx \right)^{1/\p} \right] \leq C \overline K_\mu\left(|y-z|+1\right),
\end{multline}
where the last inequality follows by Jensen's inequality. We have checked that $\tilde m_\mu$ verifies the hypothesis of the subadditive ergodic theorem (c.f.~\cite{AK}), and we obtain, for each fixed $y\in \Rd$, a random variable $\mm_\mu(y)$ such that
\begin{equation}\label{apset}
\lim_{t\to \infty} \frac1t \tilde m_\mu(ty,0) = \mm_\mu(y). 
\end{equation}
However, it turns out that $\mm_\mu(y)$ is constant $\P$--almost surely, that is, 
\begin{equation} \label{e.ergstk}
\P\Big[ \mm_\mu(y) = \E\left[ \mm_\mu(y) \right]  \Big] = 1. 
\end{equation}
This follows from the ergodic hypothesis and the fact that $\overline m_\mu(y)$ is invariant under translations. To see this, we write $\tilde m_\mu(y,z,\omega)$ and $\overline m_\mu(y,\omega)$ to denote dependence on $\omega\in\Omega$ and observe that, for every $z\in \Rd$, 
\begin{align*} \label{}
\mm_\mu(y, \tau_z \omega)  & = \lim_{t\to \infty} \frac1t \tilde m_\mu(ty+z,z, \omega) \\ 
& \leq \lim_{t\to \infty} \frac1t \Big(  \tilde m_\mu(ty+z,ty, \omega) + \tilde m_\mu(ty,0, \omega) + \tilde m_\mu(0,z, \omega) \Big)  \\ 
& \leq \lim_{t\to \infty} \frac1t  \tilde m_\mu(ty,0, \omega) + \limsup_{t\to\infty}  \frac1t\Big(  \osc_{B_{|z|+1}(ty)} m_\mu(\cdot,ty, \omega) + \osc_{B_{1}(0)} m_\mu(\cdot,z, \omega) \Big)  \\ 
& = \lim_{t\to \infty} \frac1t  \tilde m_\mu(ty,0, \omega)   = \mm_\mu(y, \omega).
\end{align*}
Here we used stationary, followed by~\eqref{e.subadd}, the definition of $\tilde m_\mu$ and~\eqref{l.oscbnde}. We deduce that $\mm_\mu(y,\tau_z\omega) = \mm_\mu(y,\omega)$ for all $\omega\in \Omega$ and $z\in \Rd$, which, in view of~\eqref{e.erg}, implies that each of the events $\{ \omega\in\Omega \,:\, \mm_\mu(y,\omega) > \E \left[ \mm_\mu(y,\cdot)\right] \}$ and $\{ \omega\in\Omega \,:\, \mm_\mu(y,\omega) < \E \left[ \mm_\mu(y,\cdot)\right] \}$ has probability either zero or one. So both must be of zero probability and~\eqref{e.ergstk} holds. 

We henceforth identify $\overline  m_\mu(y)$ and the deterministic quantity $\E \left[ \overline  m_\mu(y,\cdot) \right]$. With this identification, we may combine~\eqref{apset} and~\eqref{e.ergstk} to write
\begin{equation} \label{e.st1}
\P\left[ \limsup_{t\to \infty} \left| \frac{\tilde m_\mu(ty,0)}{t}-\mm_\mu(y) \right|  =0\right] = 1. 
\end{equation}
This holds for all $y\in \Rd$. By intersecting the events in~\eqref{e.st1} over all $y\in\Qd$, we get
\begin{equation} \label{e.st2}
\P\left[ \forall y\in \Qd, \ \limsup_{t\to \infty} \left| \frac{\tilde m_\mu(ty,0)}{t}-\mm_\mu(y) \right| = 0\right] = 1. 
\end{equation}

{\it Step 2.} We next verify that $\mm_\mu:\Rd\to \R$ is continuous, convex and positively homogeneous. It is immediate from~\eqref{e.mtildlip} that
\begin{equation}\label{mmuppbd}
|\mm_\mu(y)| \leq C\overline K_\mu |y|.
\end{equation}
The stationarity and subadditivity of $\tilde m_\mu$ yield that $\mm_\mu$ is sublinear. Indeed, for every $y,z\in \Rd$,
\begin{multline}\label{mmsubadd}
\mm_\mu(y+z) = \lim_{t\to \infty} \frac1t \E\left[ \tilde m_\mu(t(y+z),0 ) \right] \leq \lim_{t\to \infty} \frac1t  \E\left[ \tilde m_\mu(t(y+z),tz ) +  \tilde m_\mu(tz,0 ) \right] \\ = \lim_{t\to \infty} \frac1t  \E\left[ \tilde m_\mu(ty,0 ) \right] + \lim_{t\to \infty} \frac1t  \E\left[   \tilde m_\mu(tz,0 ) \right] 
= \mm_\mu(y) + \mm_\mu(z). 
\end{multline}
Combining~\eqref{mmuppbd} and~\eqref{mmsubadd} yields 
\begin{equation*}\label{}
\mm_\mu(y) - \mm_\mu(z) \leq \mm_\mu(y-z) \leq C\overline K_\mu |y-z|.
\end{equation*}
and by interchanging $y$ and $z$ we get
\begin{equation}\label{mmlip}
|\mm_\mu(y) - \mm_\mu(z)| \leq C\overline K_\mu |y-z|,
\end{equation}
and so $\overline m_\mu$ is Lipschitz with constant $C\overline K_\mu$. It is immediate from the form of the limit~\eqref{apset} that $\mm_\mu$ is positively homogeneous, and from this and~\eqref{mmsubadd} we deduce that $\mm_\mu$ is convex. For future reference, we observe that, $\mu \mapsto \overline m_\mu(y)$ is concave by~\eqref{e.concv}. Since this map is nondecreasing, it must also be continuous.

{\it Step 3.} We next upgrade the assertion~\eqref{e.st2} to
\begin{equation}\label{e.zeq0}
\P \left[ \forall R> 0, \ \lim_{t\to \infty} \sup_{y \in B_R} \left|  \frac1t \tilde m_\mu(ty,0, \omega) - \mm_\mu(y) \right| = 0 \right] = 1. 
\end{equation}
Observe that, for every $y\in \Rd$ and $z \in \Qd$, we have
\begin{align*}
\lefteqn{\left| \frac1t \tilde m_\mu(ty,0) - \mm_\mu(y)\right|} \qquad & \\
& \leq \frac1t\left|  \tilde m_\mu(ty,0) - \tilde m_\mu(tz,0) \right| + \left|\frac1t \tilde m_\mu(tz,0) - \overline m_\mu(z) \right| + \left|\overline m_\mu(y)-\overline m_\mu(z) \right| \\ 
& \leq \frac1t \osc_{B_{t|y-z|+2}(tz)} m_\mu(\cdot,0) + \frac1t\left| \tilde m_\mu(tz,0) - \overline m_\mu(z) \right| + C\overline K_\mu\left|y-z\right|. 
\end{align*}
Fix $R>0$. Let $\delta > 0$ and select finitely many $z_1,\ldots,z_k\in \Qd\cap B_R$ such that the union of the balls $B(z_i,\delta)$ covers $B_R$. Then from the above inequality, we find that  
\begin{multline*} \label{}
\sup_{y\in B_R} \left| \frac1t \tilde m_\mu(ty,0) - \mm_\mu(y)\right| \\ \leq \sup_{y\in B_R} \sup_{i\in \{ 1,\ldots,k\}} \frac1t \osc_{B_{t\delta+2}(tz_i)} m_\mu(\cdot,0) + \sup_{i\in \{ 1,\ldots,k\}} \frac1t\left| \tilde m_\mu(tz_i,0) - \overline m_\mu(z_i) \right| + C\overline K_\mu\delta.
\end{multline*}
Now taking the limsup as $t\to \infty$, we deduce from~\eqref{e.oscbnde} and~\eqref{e.st2} that, for every $R,\delta > 0$,
\begin{equation*} \label{}
\P \left[  \lim_{t\to \infty} \sup_{y \in B_R} \left|  \frac1t \tilde m_\mu(ty,0, \omega) - \mm_\mu(y) \right| \leq 2C\overline K_\mu \delta \right]=1.
\end{equation*}
We recover~\eqref{e.zeq0} after intersecting over all the events corresponding to $\delta\in \Q_+$ and then over all of the resulting events corresponding to $R\in \N^*$.

{\it Step 4.} We next release the vertex point using Lemma~\ref{eggg} with 
\begin{equation*} \label{}
X_{t} := \sup_{y \in B_{2R}} \left|  \frac1t \tilde m_\mu(ty,0) - \mm_\mu(y) \right|, \quad t>0.
\end{equation*}
Lemma~\ref{l.oscbnde} and~\eqref{e.zeq0} give the hypotheses of Lemma~\ref{eggg} for $X_t$, and so an application of the lemma yields, for every $R>0$,
\begin{multline*} \label{}
 \P \left[ \lim_{t\to \infty}  \sup_{y,z \in B_R} \left|  \frac1t \tilde m_\mu(ty,tz) - \mm_\mu(y-z) \right|=0\right]  \\
 \geq \P \left[ \lim_{t\to \infty} \sup_{z\in B_{R} } \sup_{y \in B_{2R}(z)} \left|  \frac1t \tilde m_\mu(ty+tz,tz) - \mm_\mu(y) \right|=0\right] =1.
\end{multline*}
Intersecting the events corresponding to $R=1,2,\ldots$, we obtain
\begin{equation} \label{e.almdone}
\P \left[ \forall R>0, \ \lim_{t\to \infty}  \sup_{y,z \in B_R} \left|  \frac1t \tilde m_\mu(ty,tz) - \mm_\mu(y-z) \right|=0\right] =1.
\end{equation}

\emph{Step 5.} We immediately obtain~\eqref{mphg}  from~\eqref{e.almdone} by the monotonicity of $\mu \mapsto  m_\mu(y,z)$, the continuity of $\mu \mapsto \overline m_\mu(y)$ (see the end of Step~2) and intersecting the events corresponding to each rational $\mu>\overline H_*$ as well as to $\mu=\overline H_*$.
\end{proof}

\begin{remark}\label{e.huppy}
For future reference we note that, for any $\beta>0$ and $\mu \geq \Lambda_1(\beta^\m + 1)$, we have $m_\mu(y,z) \geq \beta|y-z|$. Indeed, in view of the monotonicity of $\mu \mapsto m_\mu(y,z)$, it is enough to check that the cone function $\phi(y):= \beta \max\{ 0,|y-z|-1\}$ is a subsolution of~\eqref{e.maxsube} for $\mu = \Lambda_1(\beta^\m+1)$. This is easy to obtain from~\eqref{e.Hsubq}, using $|D\phi|\leq \beta$ and the fact that the diffusion term has a helpful sign due to the convexity of $\phi$. This also yields
\begin{equation} \label{}
\mu \geq \Lambda_1(\beta^\m+1) \quad \implies \quad \forall y\in\Rd, \ \ \overline m_\mu(y) \geq \beta|y|.
\end{equation}
In view of the concavity of $\mu\mapsto \overline m_\mu(y)$, which was obtained in Step~2 of the proof above, we get the following: there exists $c>0$ such that, for every $\mu \geq \nu \geq \overline H_*$ and $y,z\in \Rd$,
\begin{equation*} \label{}
 \overline m_\mu (y) \geq   \overline  m_\nu(y) + c\mu^{-(\m-1)/\m}(\mu-\nu)|y|.
\end{equation*}
\end{remark}


\section{Identification of the effective Hamiltonian}
\label{eff}

In this section, we define $\overline H$ in terms of the family $\{ \overline m_\mu \,:\, \mu \geq \overline H_*\}$ of homogenized maximal subsolutions  and proceed to study some of its basic properties. Throughout this section we assume that $\P$ is a given probability measure satisfying~\eqref{e.stat},~\eqref{e.erg} and~\eqref{e.wkcoer}.

We begin with an informal heuristic which leads to a guess for what~$\overline H$ should be, thinking in terms of an ``inverse problem." Write the metric problem at the ``theatrical scaling" by introducing a parameter $\ep > 0$ and defining
\begin{equation*} \label{}
m_\mu^\ep(x):= \ep m_\mu\left(\frac x\ep ,0\right).
\end{equation*}
At this scale, Proposition~\ref{shape} asserts that $m^\ep_\mu\rightarrow \overline m_\mu$ locally uniformly in $\Rd$ and $\P$--almost surely, as $\ep \to 0$, and we may write~\eqref{m-12} as
\begin{equation*} \label{}
-\ep \tr(A\left( \frac x\ep \right)D^2 m^\ep_\mu)+H\left(Dm^\ep_\mu, \frac x\ep \right) = \mu \quad \mbox{in}\ \Rd \setminus \overline B_\ep(0).
\end{equation*}
By formally passing to the limit $\ep \to 0$ in this equation (and in the rescaled version of~\eqref{m-11}) under the assumption that it ``homogenizes," this suggests that we should obtain
\begin{equation} \label{daccord}
\overline H\left(D\overline m_\mu\right) \leq \mu \quad \mbox{in} \ \Rd
 \quad \mbox{and}\quad \overline H\left(D\overline m_\mu\right) = \mu \quad \mbox{in} \ \Rd\setminus \{ 0 \}.
\end{equation}
That is, we expect that $\overline m_\mu$ is the maximal subsolution of $\overline H$ with respect to $\mu$ and the gradient of this positively homogeneous function should prescribe the $\mu$--level set of $\overline H$; the image of its subdifferential should be the $\mu$--sublevel set of~$\overline H$.  

In view of this discussion, we simply \emph{define} $\overline H$ in such a way that this is so:
\begin{equation} \label{Hbardef}
\overline H(p):= \inf\left\{ \mu \geq \overline H_*\,:\, \forall y\in \Rd, \ \ \overline m_\mu(y) \geq p\cdot y\right\}.
\end{equation}
Note that since $\overline m_\mu$ is convex positively homogeneous, the subdifferential $\partial m_\mu(0)$ is actually the closed convex hull of the image of $\Rd$ under $D\overline m_\mu$. Recall that the subdifferential $\partial \phi(x)$ of a convex function $\phi:\Rd\to\R$ at a point $x$ is defined by
\begin{equation*} \label{}
\partial \phi(x):= \left\{ p\in \Rd \,:\, \forall y \in \Rd,\ \phi(y) \geq \phi(x) - p \cdot (y-x) \right\}. 
\end{equation*}
We expect $\partial \overline m_\mu(0)$ to be the $\mu$--sublevel set of $\overline H$ and the image of $\Rd$ under $D\overline m_\mu$ to be the $\mu$--level set of $\overline H$. This indeed follows from~\eqref{Hbardef} and we may invert this formula to write $\overline m_\mu$ in terms of~$\overline H$:
\begin{equation} \label{mm}
\overline m_\mu(y)  = \sup\left\{ p\cdot y\,:\, \overline H(p) \leq \mu \right\}.
\end{equation}
That is, $\overline m_\mu$ is simply the support function of the $\mu$--sublevel set of $\overline H$. So the definition~\eqref{Hbardef} is formally in accord with~\eqref{daccord}, and once we have verified that $\overline H$ is convex (which we do below in Lemma~\ref{Hbarpropts}), checking the latter in the viscosity sense is simply a routine exercise. Since here we do not actually use this fact, we omit the argument, but the reader may consult for example~\cite{AS2} or else argue directly that the maximal subsolutions of a constant-coefficient convex Hamiltonian are the support functions of the sublevel sets.

We need to check that the quantity~$\overline H(p)$ is well-defined (and finite). In view of the monotonicity of~$\mu\mapsto \overline m_\mu$, we need only show that, for every~$p\in \Rd$, there exists~$\mu> \overline H_*$ sufficiently large that the graph of~$\overline m_\mu$ is above the plane~$y\mapsto p\cdot y$. But this is immediate from~\eqref{e.huppy}, which in fact gives the estimate
\begin{equation} \label{Hbarbnds}
\overline H_*\leq \overline H(p) \leq \Lambda_1\left( |p|^\m + 1\right).
\end{equation}
We collect some more basic properties of the effective Hamiltonian $\overline H:\Rd \to \R$ in the following lemma.

\begin{lem}\label{Hbarpropts}
The function $\overline H:\Rd \to \R$ is continuous, convex and there exist $C,c>0$, depending only on $d$, such that
\begin{equation} \label{Hgrth}
\barH_*=\min_{p \in \Rd} \barH(p) \qquad \mbox{and} \qquad c\overline K_0^{-\m} \left( |p| -  C\overline K_0\right)^\m \leq \overline H(p) \leq \Lambda_1\left( |p|^\m + 1\right).
\end{equation}
\end{lem}
\begin{proof}
By definition, $\overline H(\cdot) \geq\overline H_*$. On the other hand, take $\delta > 0$, set $\mu:=\overline H_*+\delta$. Since $\overline m_\mu$ is convex, we may select $p_0\in \partial m_\mu(0)$. This implies that $\overline m_\mu(y) \geq p_0\cdot y$ for every $y\in \Rd$. Thus $$\min_{p\in\Rd} \overline H(p) \leq \overline H(p_0) \leq \mu = \overline H_*+ \delta.$$ Since $\delta >0$ was arbitrary, we obtain the first assertion of~\eqref{Hgrth}. 

The upper bound for $\overline H$ was proved already in~\eqref{Hbarbnds}. The lower bound follows from~\eqref{e.Kmugrmu} and~\eqref{mmuppbd} and the definition of~$\overline H$ after an easy computation. 
\end{proof}

An immediate consequence of~the convexity of $\overline H$ is that, with the possible exception of the minimal level set $\{ \overline H = \overline H_*\}$, each of the level sets of $\overline H$ are the boundary of the corresponding sublevel set. That is, for every $p\in\Rd$,
\begin{equation} \label{e.pbndry}
\barH(p) > \barH_* \qquad \mbox{implies that} \qquad p \in \partial \left\{ \hat p\in \Rd \, : \, \barH(\hat p) \leq \barH(p) \right\}.
\end{equation}
To prove the main homogenization result, we need further geometric information, summarized in the following lemma, relating the level sets of $\overline H$ and the maximal subsolutions.

\smallskip

Recall that if~$K\subseteq \Rd$ is closed and convex, an \emph{exposed point} of~$K$ is a point $p\in K$ such that there exists a linear functional~$l:\Rd \to \R$ such that~$l(p)>l(\hat p)$ for every~$\hat p \in K\setminus \{p \}$. The set of exposed points are, for a general bounded convex subset~$K$ of~$\Rd$, a proper subset of the set of extreme points of~$K$. However, Straszewicz's theorem (c.f.~\cite[Theorem 18.6]{Rock}) asserts that every extreme points is a limit of exposed points. 

\begin{lem} \label{geom}
Let $\mu \geq \barH_*$ and $p \in \partial \left\{ \hat p\in \Rd\,:\, \barH(\hat p) \leq \mu \right\}$. Then there exists a unit vector $e\in \partial B_1$ such that 
\begin{equation}\label{touch}
\mm_\mu (e) -p\cdot e = 0 = \inf_{y\in \Rd} \left( \mm_\mu (y) -p\cdot y \right).
\end{equation}
If in addition $p$ is an exposed point of $\left\{ \hat p\in \Rd\,:\, \barH(\hat p) \leq \mu \right\}$, then $e$ can be chosen in such a way that $\mm_\mu$ is differentiable at $e$ with $p=D\mm_\mu(e)$.
\end{lem}
\begin{proof}
Set $S:= K_\mu=\left\{ \hat p \in \Rd \,: \, \barH(\hat p) \leq \mu \right\}$. By elementary convex separation, there exists a linear functional $l:\Rd \to \R$ such that $l(p) = 0$ and $l(\hat p) \leq 0$ for every $\hat p\in  S $. If $p$ is an exposed point, then we also take $l$ so that $l(\hat p) < 0$ for every $\hat p\in S \setminus \{ p \}$. According to the representation theorem, there exists $e\in \Rd$ such that $l(x) = e\cdot(x-p)$. By normalizing, we may assume that $|e|=1$. We deduce that, for every $y\in \Rd$,
\begin{equation}\label{dumbass}
\mm_\mu(e) - p\cdot e = \sup\left\{ (\hat p-p)\cdot e \, : \, \hat p\in S \right\} = 0 \leq \sup\left\{ (\hat p-p)\cdot y \, : \, \hat p\in S \right\} = \mm_\mu(y) - p\cdot y.
\end{equation}
This is \eqref{touch}. Since $\mm_\mu$ is positively homogeneous, we see that $p\in \partial \mm_\mu(e)$. In fact, what~\eqref{dumbass} shows is precisely that
\begin{equation}\label{}
\partial  \mm_\mu(e) = \left\{ \hat p\in S\,:\, l(\hat p) = 0\right\}.
\end{equation}
Thus if $p$ is an exposed point of $S$, then we have $\partial \mm_\mu(e)= \{ p \}$ by our choice of~$l$. This implies that $\mm_\mu$ is differentiable at $e$ and $D\mm_\mu(e)=p$. 
\end{proof}

\begin{remark}
We can express~$\overline H$ via the following ``min--max" formula:
\begin{multline} \label{e.minmax}
\overline H(p) = \inf \bigg\{ \mu\in \R \,:\,  \mbox{there exists} \ \  w\in C^{0,1}_{\mathrm{loc}} (\Rd) \ \ \mbox{satisfying} \ \ \eqref{e.maxsube} \\ \mbox{and} \ \ \liminf_{|y| \to\infty } \frac{w(y)-p\cdot y}{|y|} \geq 0  \bigg\}.
\end{multline}
Indeed, if $w \in \USC(\Rd)$ satisfies~\eqref{e.maxsube}, then 
\begin{equation*} \label{}
\overline m_\mu(y) -p\cdot y \geq \liminf_{t \to \infty} \frac{w(ty)-p\cdot (ty)}{|ty|}.
\end{equation*}
If the latter is nonnegative for all $y\in \Rd$, then $\overline H(p) \leq \mu$ by definition. This yields~ ``$\leq$" in~\eqref{e.minmax}. To obtain the reverse inequality, we use $m_\mu$ with $\mu=\overline H(p)$ as the witness and observe that 
\begin{align*}
\liminf_{|y| \to \infty} \frac{m_\mu(y) - p\cdot y}{|y|} = \liminf_{t\to\infty} \inf_{|y| = 1} \left( \frac{m_\mu(ty)}{t} - p\cdot y \right) = \inf_{|y| = 1} \left( \overline m_\mu(y) - p\cdot y\right) \geq 0. 
\end{align*}
The reason that we call~\eqref{e.minmax} a ``min--max" representation is that it can be formally written
\begin{equation} \label{e.minmax2}
\overline H(p) = \inf_{w \in \mathcal L_p} \sup_{y\in \Rd} \left(-\tr \left( A(y) D^2w(y) \right) + H( Dw(y),y) \right),
\end{equation}
where
\begin{equation*}
\mathcal L_p:=\left\{ w\in C^{0,1}_{\mathrm{loc}} (\Rd) \  : \ \liminf_{|y| \to\infty } \frac{w(y)-p\cdot y}{|y|} \geq 0 \right\}.
\end{equation*}
The expression inside the infimum on the right of~\eqref{e.minmax2} does not make sense, due to the fact that~$w$ may not have enough regularity. It must therefore be interpreted in the viscosity sense, and this leads precisely to~\eqref{e.minmax}.
\end{remark}

\section{Homogenization of the approximate cell problem}\label{MAC}

In this section, we show using a comparison argument that Proposition~\ref{shape} implies an homogenization result for a special time--independent problem. The particularities of this argument are new here, even for uniformly coercive Hamiltonians or first-order equations. 

\smallskip

Throughout we assume~$\P$ is a probability measure on~$(\Omega,\F)$ satisfying~\eqref{e.stat},~\eqref{e.erg} and~\eqref{e.wkcoer}.

\smallskip

For each fixed $p\in\Rd$, we consider the problem
\begin{equation}\label{macro}
w^\ep -\ep \tr\left (A\left (\frac{x}{\ep}\right)D^2 w^\ep\right) +H\left(p+Dw^\ep,\frac x\ep\right)=0 \quad \mbox{in}\ \R^d.
\end{equation}
As we will see below,~\eqref{macro} has a unique bounded-below solution with probability one which we denote by $w^\ep(\cdot,p)$. We argue that
\begin{equation}\label{macro-1}
\P \left[ \forall p\in \Rd, \ \forall R>0,  \  \  \limsup_{\ep \to 0} \sup_{x\in B_R} \left| w^\ep(x,p) + \barH(p) \right| = 0 \right] = 1. 
\end{equation}
Recall that~\eqref{macro}, often written at a different scale than~\eqref{macro} (see~\eqref{vepeq} below), is often called the \emph{approximate cell problem} and homogenizing it (by which we mean proving~\eqref{macro-1}) is the key step in the derivation of Theorem~\ref{Hg} from Proposition~\ref{shape}. To see why we expect $w^\ep(\cdot,p)$ to converge locally uniformly to the constant $-\overline H(p)$ as $\ep \to 0$, observe that the (unique) solution of
\begin{equation} \label{macro-homed}
w + \overline H(p+Dw) = 0\quad \mbox{in} \ \Rd
\end{equation}
is precisely the constant function $w\equiv -\overline H(p)$. Thus~\eqref{macro-1} can be understood roughly as the assertion that ``\eqref{macro} homogenizes to~\eqref{macro-homed}."

\subsection{Basic properties of~\eqref{macro}}
In order to prove~\eqref{macro-1}, we must first establish some basic properties of~\eqref{macro} including addressing the question of well-posedness. In the uniformly coercive case, it is straightforward to show (and classical) that the Perron method and the comparison principle yield a unique bounded solution of~\eqref{macro} given by the formula
\begin{equation} \label{e.wepdef}
w^\ep(x,p) := \sup \left\{ v(x)\,: \, v \in \USC(\Rd) \ \mbox{is a subsolution of~\eqref{macro}} \right\}.
\end{equation}
Well-posedness in the general weakly coercive setting is more nontrivial because it is less easy to show \emph{a priori} that $w^\ep(\cdot,p)$ satisfies a suitable growth condition at infinity for the application of the comparison principle. 

We take~\eqref{e.wepdef} to be the \emph{definition} of the function~$w^\ep(x,p)$ and continue with a discussion of some elementary properties of $w^\ep$. First, we remark that it is often convenient to consider~\eqref{macro} at the microscopic scale, in order to use the stationary of the environment. The rescaled equation is 
\begin{equation} \label{vepeq}
\ep v - \tr\left (A(y) D^2 v\right) +H\left(p+Dv,y\right)=0 \quad \mbox{in}\ \R^d.
\end{equation}
and we rescale $w^\ep$ by introducing 
\begin{equation} \label{weprsc}
v^\ep(y,p):= \frac1\ep w^\ep(\ep y,p) = \sup \left\{ v(x)\,: \, v \in \USC(\Rd) \ \mbox{is a subsolution of~\eqref{vepeq}} \right\}.
\end{equation}
The second equality in~\eqref{weprsc} follows from the definition of~$w^\ep$ and a rescaling of~\eqref{macro}. Note that it is immediate from~\eqref{weprsc} that $v^\ep(x,p)$ is stationary with respect to the translation action. According to~\cite[Theorem 6.1]{AT}, for every $\ep > 0$, $p\in \Rd$ and choice of coefficients $(\sigma,H) \in \Omega$, the function $v^\ep(\cdot,p)$ defined in~\eqref{weprsc} belongs to $C^{0,1}_{\mathrm{loc}}(\Rd)$ and is a solution of~\eqref{vepeq}. It follows immediately from reversing the scaling that $w^\ep(\cdot,p)\in C^{0,1}_{\mathrm{loc}}(\Rd)$ is a solution of~\eqref{macro}. Uniqueness is a separate issue addressed below, see~\eqref{macro-uniq}.

\smallskip

Next, we observe that~$w^\ep(\cdot,p)$ is bounded below uniformly in $\ep$. Indeed, for all $p\in\Rd$,
\begin{equation} \label{wepbb}
\inf_{x\in \Rd} w^\ep(x,p) \geq -\Lambda_1\!\left( |p|^\m +1 \right). 
\end{equation}
This follows from the definition of $w^\ep$ and the fact that the right side of this inequality is a subsolution of~\eqref{macro}, according to~\eqref{e.Hsubq}, as we have already seen in Remark~\ref{e.huppy}. Using this bound for the equation at the microscopic scale, we obtain that $v^\ep(\cdot,p)$ is a solution of the inequality
\begin{equation*} \label{}
 - \tr\left (A(y) D^2 v^\ep\right) +H\left(p+Dv^\ep,y\right) \leq \Lambda_1\!\left( |p|^\m +1 \right) \quad \mbox{in}\ \R^d.
\end{equation*}
Then according to the definition of $m_\mu$ with $\mu =  \Lambda_1\left( |p|^\m +1 \right)$, we obtain the estimate 
\begin{equation} \label{e.entrape}
 v^\ep( y, p) - \sup_{x\in B_1(z)} v^\ep( x, p)\leq m_\mu(y,z) \qquad \mbox{for every} \ \ \mu \geq \Lambda_1\!\left( |p|^\m +1 \right).
\end{equation}
Note that this inequality holds uniformly in $\ep$. 

\begin{lem} \label{wepconca}
For every $\ep > 0$, $x\in \Rd$ and $(\sigma,H) \in \Omega$,
\begin{equation} \label{wepconc}
p\mapsto w^\ep(x,p) \quad \mbox{is concave.}
\end{equation}
\end{lem}
\begin{proof}
Observe that if~$v_1, v_2 \in \USC(\Rd)$ are subsolutions of~\eqref{macro} with~$p=p_1$ and~$p=p_2$, respectively, and $\lambda\in [0,1]$, then the function~$\lambda v_1 + (1-\lambda) v_2$ is a subsolution of~\eqref{macro} with~$p= \lambda p_1 + (1-\lambda) p_2$. This follows formally from the convexity of the Hamiltonian, and for a rigorous proof we refer to the argument of~\cite[Lemma 2.4]{AT}. In view of the definition of~$w^\ep$ in~\eqref{e.wepdef}, this observation gives the lemma.
\end{proof}

An immediate consequence of~\eqref{wepbb} and Lemma~\ref{wepconca} is that, for every $k>0$, the map $p\mapsto \max\{ k, w^\ep(x,p) \}$ is uniformly continuous. Indeed, we obtain that, for all $p,\hat p\in\Rd$ with $|p-\hat p| < 1$,
\begin{equation} \label{wepusc}
w^\ep(x,p) \geq \big( 1 - |p-\hat p| \big) w^\ep(x,\hat p) - \Lambda_1\big( |p|^\m +1 \big) |p-\hat p|.
\end{equation}

We next show that $w^\ep(x,p)$ satisfies, almost surely with respect to $\P$, an appropriate sublinear growth condition uniformly in $\ep$ and for bounded $|p|$. This is required both in order to establish $w^\ep$ as the \emph{unique} bounded-below solution of~\eqref{macro} and is also needed in the proof of~\eqref{macro-1}. Note that this estimate is trivial for uniformly coercive Hamiltonians, since in this case $w^\ep(x,p)$ is bounded above uniformly for $x\in \Rd$, $p\in B_R$ and $0 < \ep \leq 1$. In the general case, it is a consequence of the averaged coercivity condition~\eqref{e.wkcoer} and its proof uses the ergodic theorem, which is the reason we expect it to hold only almost surely with respect to $\P$.

\begin{lem} \label{GC}
We have
\begin{equation} \label{e.wepgc}
\P\left[ \forall R> 0, \ \limsup_{|x|\to\infty} \sup_{|p| \leq R} \sup_{0<\ep\leq 1} \frac{|w^\ep(x,p)|}{|x|} = 0 \right] = 1. 
\end{equation}
\end{lem}
\begin{proof}
In view of~\eqref{wepbb}, we need only prove upper bounds for $w^\ep$.  For most of the argument we work at the microscopic scale and we split the proof into four steps. It clearly suffices to prove the lemma for fixed $R>0$, since we obtain the general case by intersecting the events corresponding to rational~$R$. 

It is convenient to work with the random fields
\begin{equation*} \label{}
V^\ep(y):= \sup_{|p| \leq R} \sup_{z\in B_1(y)}  v^\ep(z,p).
\end{equation*}
Note that $V^\ep$ is stationary with respect to the translation group action. According to~\cite{AT}, the family $\{ V^\ep\}_{\ep>0}$ is locally equi-Lipschitz continuous in $\Rd$ for every realization $\omega=(\sigma,H)\in\Omega$ of the coefficients.

\emph{Step 1.}
We begin from the estimate from~\cite{AT} that, for $C>0$ depending only on $d$ and $\m$,
\begin{equation} \label{e.barrests}
 \ep V^\ep(0) \leq M_2 \left(1+\Lambda_1R^\m \right) +  C\left( \frac{\Lambda_2}{a_2} \right)^{1/(\m-1)}.
\end{equation}
This is shown by an explicit computation using smooth test functions, see \cite[Section 4]{AT}. Let~$\xi$ denote the random variable on the right side of~\eqref{e.barrests} and $I$ denote its essential infimum (with respect to $\P$):
\begin{equation*} \label{}
I: = \inf\left\{ \lambda \in \R \,:\, \P \left[ \xi < \lambda \right] > 0 \right\} < \infty.
\end{equation*}
We eventually apply Lemma~\ref{eggg} to the sequence of random fields defined by 
\begin{equation*} \label{}
X_t(y) := \frac1t \inf_{z\in B_t(y)} \sup_{0<\ep\leq 1} \left( V^\ep(z) - \frac2\ep I \right), \quad t>0. 
\end{equation*}
In the next few steps we check that the hypotheses of Lemma~\ref{eggg} hold for $X_t$.

\emph{Step 2.} We show that  
\begin{equation} \label{verrgo1}
\P \left[ \limsup_{t\to \infty} X_t(0) \leq 0 \right]  = 1.
\end{equation}
According to the ergodic theorem,
\begin{equation*} \label{}
\P \left[ \lim_{s\to \infty} \fint_{B_s} \indc_{\{\xi(\cdot) \leq 2I\}}(y) \, dy = \P \left[ \xi(0) \leq 2I  \right] \right] = 1,
\end{equation*}
Here $\indc_E : \Rd \to \R$ is the characteristic function of a Borel set $E\subseteq \Rd$. Note that $\P \left[ \xi(0) \leq 2I  \right] >0$ by the definition of~$I$ and that, if $\indc_{\{\xi(\cdot) \leq 2I\}}(y)$ does not vanish identically in $B_t$, then $X_t \leq 0$ by~\eqref{e.barrests}. This yields~\eqref{verrgo1}. 

\emph{Step 3.} We show that
\begin{equation} \label{verrgo2}
\P \left[  \limsup_{r\to 0} \limsup_{t\to \infty} \frac1t \sup_{y\in B_t} \sup_{0<\ep\leq 1} \osc_{B_{rt}(y)}  V^\ep = 0 \right] =1. 
\end{equation}
To see this, observe that~\eqref{e.entrape} implies that, for every $\ep > 0$ and $y,z\in\Rd$,
\begin{equation*} \label{}
V^\ep(y) - V^\ep(z) \leq  \tilde m_\mu(y,z) \qquad \mbox{with} \ \ \mu := \Lambda_1\!\left( R^m +1 \right).
\end{equation*}
We therefore obtain~\eqref{verrgo2} from~\eqref{e.oscbnde}. As a consequence of~\eqref{verrgo2}, we get
\begin{equation} \label{verrgo3}
\P \left[  \limsup_{r\to 0} \limsup_{t\to \infty} \sup_{y\in B_t} \osc_{B_{rt}(y)} X_t = 0 \right] =1. 
\end{equation}

\emph{Step 4.} We complete the argument. In view of~\eqref{verrgo1} and~\eqref{verrgo3}, we may apply Lemma~\ref{eggg} to conclude that
\begin{equation*} \label{}
\P \left[ \forall K>0, \ \limsup_{t\to \infty} \sup_{y\in B_{Kt}} X_t(y) \leq 0 \right]  =1. 
\end{equation*}
Using the definition of $X_t$, replacing $Kt$ by $t$ and setting $r=1/K$, this gives
\begin{equation*} \label{}
\P \left[ \forall r>0, \ \limsup_{t\to \infty}\frac1t \sup_{y\in B_{t}} \inf_{z\in B_{rt}(y)} \sup_{0< \ep\leq 1} \left( V^\ep(z) -\frac2\ep I \right) \leq 0 \right]  =1. 
\end{equation*}
Using again~\eqref{verrgo2}, we obtain
\begin{equation*} \label{}
\P \left[  \limsup_{t\to \infty}\frac1t \sup_{y\in B_{t}}  \sup_{0< \ep\leq 1} \left( V^\ep(y) -\frac2\ep I \right) \leq 0 \right]  =1. 
\end{equation*}
Using the definition of $V^\ep$ and rewriting the expression in terms of $w^\ep$, we get
\begin{equation*} \label{}
\P \left[ \limsup_{t\to \infty} \sup_{0<\ep\leq 1} \sup_{|p| \leq R} \sup_{x\in B_{\ep t}} \frac{w^\ep(x,p)-2I}{\ep t}  \leq 0 \right]  =1. 
\end{equation*}
This is actually stronger than~\eqref{e.wepgc}. Indeed:  
\begin{align*}
\limsup_{t\to \infty}\sup_{0<\ep\leq 1} \sup_{|p| \leq R} \sup_{x\in B_{\ep t}} \frac{w^\ep(x,p)-2I}{\ep t} &  = \lim_{s\to \infty} \sup_{t\geq s} \sup_{0<\ep\leq 1}\sup_{|p| \leq R} \sup_{x\in B_{\ep t}} \frac{w^\ep(x,p)-2I}{\ep t} \\
& \geq \limsup_{s\to \infty} \sup_{0<\ep\leq 1}\sup_{|p| \leq R} \sup_{x\in B_{s}} \frac{w^\ep(x,p)-2I}{s}\\
& \geq \limsup_{|x| \to \infty} \sup_{0<\ep\leq 1}\sup_{|p| \leq R} \frac{w^\ep(x,p)}{|x|}.
\end{align*}
Note that the inequality on the second line was obtained by reversing the first two supremums and then taking~$t=s/\ep$ in the supremum over~$t$. This completes the proof. 
\end{proof}

It follows from Lemma~\ref{GC} and \cite[Theorem 2.1]{AT} that, with probability one, $w^\ep(\cdot,p)$ is the unique bounded-below solution of~\eqref{macro} for every fixed $\ep> 0$ and $p\in\Rd$. That is:
\begin{multline} \label{macro-uniq}
\P \Big[ \forall p\in\Rd,\ \forall \ep > 0: \ w^\ep(\cdot,p) \ \mbox{belongs to $C^{0,1}_{\mathrm{loc}}(\Rd)$ and is the unique solution} \\ \mbox{ of~\eqref{macro}   which is bounded below on $\Rd$} \Big] = 1. 
\end{multline}

\subsection{The proof of~\eqref{macro-1}}

The next lemma is the first step in the direction of~\eqref{macro-1}. For the argument we again use Lemma~\ref{eggg}.

\begin{lem}\label{flatspot}
We have
\begin{equation} \label{flatspoke}
\P \left[ \forall p\in \Rd, \ \forall  R>0, \ \limsup_{\ep \to 0} \sup_{x\in B_R} w^\ep(x,p) \leq -\overline H_* \right] = 1.
\end{equation}
\end{lem}
\begin{proof}
Here we employ a soft compactness argument using the rescaled functions $v^\ep$ defined in~\eqref{weprsc}. Denote the event 
\begin{equation*} \label{}
E:= \Bigg\{ (\sigma,H) \in \Omega \,:\, \overline H_* = \inf\left\{ \mu \in \R\,:\,  \mbox{there exists} \ \ w \in \USC(\Rd) \ \mbox{satisfying~\eqref{e.maxsube}}  \right\} \Bigg\}.
\end{equation*}
Recall from Lemma~\ref{l.detcon} that $\P[E] = 1$.

\emph{Step 1.} We first show that, for all $p\in \Rd$ and $\omega\in E$,
\begin{equation}\label{step3-1}
\limsup_{\ep \to 0} \sup_{z\in B_1} \ep v^\ep(z,p) \le -\barH_*.
\end{equation}
Suppose on the contrary that there exist $\eta>0$ and a subsequence $\ep_k \to 0$ such that,  for every $k\in\N$, $${\ep_k} \sup_{z\in B_1} v^{\ep_k}(z,p) \geq -\overline H_* + \eta.$$
Define the function 
\begin{equation*} \label{}
\tilde v^\ep(y,p):= p\cdot y + v^\ep(y,p) - \sup_{z\in B_1} v^\ep(z,p).
\end{equation*}
According to the local Lipschitz estimates~\cite[Proposition 3.1]{AT} and~\eqref{e.wepgc}, the family $\{ \tilde v^\ep \}_{\ep > 0}$ is uniformly bounded in $C^{0,1}(B_s)$ for every $s>0$. By taking a further subsequence of $\{ \ep_k\}$, we may suppose that $\tilde v^{\ep_k}$ converges locally uniformly on $\Rd$ to a  function $\tilde v\in C^{0,1}_{\mathrm{loc}}(\Rd)$. In view of the fact that $\tilde v^\ep$ satisfies the equation
\begin{equation*} \label{}
\ep \tilde v^\ep - \tr\left (A(y) D^2 \tilde v^\ep\right) +H\left(D\tilde v^\ep,y\right)= -\ep \sup_{z\in B_1} v^\ep(z,p) \quad \mbox{in}\ \R^d,
\end{equation*}
we obtain, by the stability of viscosity solutions under local uniform convergence, that $\tilde v$ satisfies
\begin{equation*} \label{}
- \tr\left (A(y) D^2 \tilde v \right) +H\left(D\tilde v,y\right) \leq \overline H_* - \eta \quad \mbox{in}\ \R^d.
\end{equation*}
This contradicts the assumption that $\omega=(\sigma,H) \in E$ and completes the proof of~\eqref{step3-1}. As a consequence, we obtain that
\begin{equation} \label{step3-1q}
\P\left[ \forall p\in \Rd, \ \limsup_{\ep \to 0} \sup_{z\in B_1} \ep v^\ep(z,p) \le -\barH_* \right] = 1. 
\end{equation}

\emph{Step 2.}
To obtain the conclusion of the lemma from~\eqref{step3-1}, we apply Lemma~\ref{eggg} to the family of random variables
\begin{equation*} \label{}
X_t:= \sup_{z\in B_1} \ep v^\ep(z,p),\quad \mbox{with}  \  t = \ep^{-1}.
\end{equation*}
The first hypothesis of Lemma~\ref{eggg} is satisfied by~\eqref{step3-1} and the second hypothesis is confirmed by~\eqref{e.entrape} and~\eqref{e.oscbnde}. The conclusion of Lemma~\ref{eggg} yields that, for every $p\in \Rd$,
\begin{equation*} \label{}
\P\left[ \forall R>0, \ \limsup_{\ep \to 0} \sup_{z\in B_{R/\ep}} \ep v^\ep(z,p) \le -\barH_* \right] = 1. 
\end{equation*}
Using~\eqref{wepusc} and intersecting over all events corresponding to rational $p$, we obtain
\begin{equation*} \label{}
\P\left[ \forall p\in\Rd, \ \forall R>0, \ \limsup_{\ep \to 0} \sup_{z\in B_{R/\ep}} \ep v^\ep(z,p) \le -\barH_* \right] = 1. 
\end{equation*}
This is equivalent to~\eqref{flatspoke}.  
\end{proof}

We now show that~\eqref{macro} homogenizes to~\eqref{macro-homed}.

\begin{prop}
\label{p.SH}
The assertion~\eqref{macro-1} holds.
\end{prop}
\begin{proof}

The argument is deterministic and based on the comparison principle. To give an overview of the proof, we introduce the following  events:
\begin{equation*} \label{}
E_1:= \left\{ (\sigma,H) \in \Omega \,:\, \forall \mu \geq \overline H_*, \ \forall R>0, \ \limsup_{t\to \infty} \sup_{y,z\in B_R}\left| \frac{m_\mu(ty,tz)}t -  \mm_\mu(y-z) \right| = 0 \right\},
\end{equation*}
\begin{equation*} \label{}
E_2:= \Bigg\{ (\sigma,H) \in \Omega \,:\, \forall R> 0, \ \limsup_{|x|\to\infty} \sup_{|p| \leq R} \sup_{0<\ep\leq 1} \frac{|w^\ep(x,p)|}{|x|} = 0 \Bigg\},
\end{equation*}
\begin{equation*} \label{}
E_3:= \Bigg\{ (\sigma,H) \in \Omega \,:\,\forall p\in \Rd, \ \limsup_{\ep \to 0} \sup_{x\in B_R} w^\ep(x,p) \leq -\overline H_* \Bigg\}
\end{equation*}
and finally
\begin{equation*} \label{}
E_4:= \left\{ (\sigma,H) \in \Omega \,:\, \forall p\in \Rd, \ \forall R>0, \  \limsup_{\ep \to 0} \sup_{x\in B_R} \left| w^\ep(x,p) +\overline H(p) \right| = 0 \right\}.
\end{equation*}
According to~Proposition~\ref{shape}, Lemma~\ref{GC} and Lemma~\ref{flatspot}, we have 
\begin{equation*} \label{}
\P\big[ E_1\cap E_2 \cap E_3 \big] =1.
\end{equation*}
To obtain $\P[E_4]=1$, it therefore suffices to demonstrate that
\begin{equation} \label{SH.wts}
E_1 \cap E_2 \cap E_3 \subseteq E_4. 
\end{equation}
Thus for the remainder of the proof we fix $p\in \Rd$, $R>0$ and $(\sigma,H) \in E_1\cap E_2\cap E_3$  and argue that 
\begin{equation} \label{e.bleckt}
\limsup_{\ep \to 0} \sup_{x\in B_R} \left| w^\ep(x,p) +\overline H(p) \right| = 0.
\end{equation}
The proof of~\eqref{e.bleckt} is broken into two steps.

\smallskip

{\it Step 1.} We show that
\begin{equation}\label{step1}
\liminf_{\ep \to 0} \inf_{z\in B_R} w^\ep(z,p) \geq -\overline H(p).
\end{equation}
We begin with some reductions. By the concavity of the map $\hat p\mapsto w^\ep(x,\hat p)$, we may assume without loss of generality that $p$ is an extreme point of $\left\{ \hat p\,:\, \barH(\hat p) \leq \barH(p) \right\}$. Second, by~\eqref{wepusc}, we may also suppose that $\barH(p) > \barH_*$. Next, Straszewicz's theorem~\cite[Theorem~18.6]{Rock} and~\eqref{wepusc} permit us to further suppose that $p$ is an \emph{exposed} point of $\left\{ \hat p\,:\, \barH(\hat p) \leq \barH(p) \right\}$. This is useful in view of~\eqref{e.pbndry} and Lemma~\ref{geom}, which imply the existence of $e \in \partial B_1$ such that $\mm_{\mu}(e)=e \cdot p$ and $\mm_\mu$ is differentiable at $e$ with $p = D \mm_{\mu}(e)$, where as usual we have set $\mu:=\barH(p)$ for convenience. In view of the limit~\eqref{mphg}, this forces the function $m_\mu(\cdot,z-te)$, with $t>0$ very large, to be very ``flat" in large balls centered at $z$, as we will see. This is what allows us to use this function as an ``approximate subcorrector" in order to bound $w^\ep$ from below. 

We proceed with the demonstration of \eqref{step1} by supposing that $- \overline H(p) - w^\ep(z,p) \geq \delta > 0$ for some~$z\in B_R$ and deriving a contradiction if $0< \ep \leq 1$ is too small. The idea is to compare~$w^\ep(\cdot,p)$ in the ball $B_s(z)$, for a large enough but fixed $s>0$, to the function $x\mapsto -p\cdot (x-z+te) + \ep m_\mu(x/\ep,(z-te)/\ep)$ for $t\gg s$. We argue that the former is a strict supersolution of the equation solved by the latter, and then we derive a contradiction by showing that their difference has a local minimum. To ensure that we can touch the first function from below by the second, we use the fact that both functions are expected to be ``flat" near~$z$ (for the second function, this is due to the fact that~$p=D\mm_\mu(e)$), and we add a small linearly growing perturbative term made possible by the positivity of $\delta$. 

In order to prepare $w^\ep(\cdot,p)$ for comparison, we take $c>0$  and $\lambda >1$ to be selected below and define the auxiliary function
\begin{equation*}\label{}
W^\ep(x):= \lambda \left ( w^\ep(x,p) - w^\ep(z,p) \right )+ c \delta \left(  \left( 1 + |x-z|^2 \right)^{1/2} - 1\right).
\end{equation*}
Since~$\omega\in E_2$, there exists an $s>0$, which does not depend on $z$ or $\ep > 0$, such that 
\begin{equation*}\label{}
U_\ep := \left\{ x\in \Rd \,:\, W^\ep(x) \leq \frac14\delta \right\} \subseteq B_s(z)
\end{equation*}
We claim that, by choosing~$\lambda$ sufficiently close to~$1$ and~$c>0$ sufficiently small depending on $\lambda$, then we have
\begin{equation}\label{step1v}
 - \tr\left(A\left( \frac x\ep \right) D^2W^\ep \right) + H\left( p+DW^\ep,\frac x\ep \right) \geq \overline  H(p) + \frac12 \delta \quad \mbox{in} \ U_\ep. 
\end{equation}
In order to verify \eqref{step1v}, take any smooth test function $\varphi$ such that $v^\ep-\varphi$ has a strict local minimum at $x_0\in U_\ep$. Set $\psi(x):=\left( 1 + |x-z|^2 \right)^{1/2}$.
Then $w^\ep-\lambda^{-1}(\varphi+c\delta \psi)$ has a strict local minimum at $x_0$. Using the equation satisfied by $w^\ep$ and the definition of viscosity supersolution, we obtain
\begin{equation*}
w^\ep(x_0) -\ep \tr\left(A\left( \frac {x_0}{\ep} \right)\lambda^{-1} D^2(\varphi+c \delta \psi)(x_0) \right) + H\left( p+\lambda^{-1} D(\varphi+c\delta \psi)(x_0),\frac {x_0}{\ep} \right) \geq 0.
\end{equation*}
The convexity of $H$ gives
\begin{multline*}
H\left( p+\lambda^{-1} D(\varphi+c\delta \psi)(x_0),\frac {x_0}{\ep} \right) \leq \lambda^{-1}H\left( p+D\varphi(x_0),\frac {x_0}{\ep} \right) \\
+(1-\lambda^{-1}) H\left( p+(\lambda-1)^{-1} c\delta D\psi(x_0),\frac {x_0}{\ep} \right).
\end{multline*}
Combining the above computations and using $x_0\in U_\ep$, we deduce that, for $\lambda$ sufficiently close to~$1$ and~$c>0$ sufficiently small depending on $\lambda$,
\begin{equation*}
 - \tr\left(A\left( \frac {x_0}\ep ,\omega\right) D^2\varphi(x_0) \right) + H\left( p+D\varphi(x_0),\frac {x_0}\ep,\omega \right) \geq \overline  H(p) + \frac12 \delta. 
\end{equation*}
This completes the proof of \eqref{step1v}.

We may now apply the comparison principle (c.f.~\cite[Theorem 2.2]{AT}) to conclude that, for every $t \geq s+1$,
\begin{multline}\label{step1cp}
\inf_{ x\in U_\ep} \left( W^\ep(x) + p\cdot(x - z +te) - \ep m_\mu\left(\frac{x}{\ep},\frac{z-te}{\ep}\right) \right) \\ = \inf_{x\in \partial U_\ep} \left( W^\ep(x) + p\cdot(x-z+te) - \ep m_\mu\left(\frac{x}{\ep},\frac{z-te}{\ep}\right)  \right).
\end{multline}
Estimating the infimum on the left side of \eqref{step1cp} by taking $x=z$ and recalling that $W^\ep(z) = 0$, and the term on the right side by using that $W^\ep \equiv \delta/4$ on $\partial U_\ep$ and $\partial U_\ep \subseteq B_s(z)$, we conclude after a rearrangement that, for every $t\geq s+1$,
\begin{equation}\label{step1ag}
 \inf_{x\in B_s(z)}\left( p\cdot(x-z) + \ep m_\mu\left(\frac{z}{\ep},\frac{z-te}{\ep}\right) - \ep  m_\mu\left(\frac{x}{\ep},\frac{z-te}{\ep}\right) \right) \leq -\frac14\delta.
\end{equation}
This holds for every $z\in B_R$ and $\ep> 0$ for which $- \overline H(p) - w^\ep(z,p) \geq \delta > 0$. Therefore, if $- \overline H(p) - w^{\ep_j}(z_j,p) \geq \delta$ along subsequences $\{ z_j \}_{j\in \N} \subseteq B_R$ and $\ep_j\to 0$, then by passing to limits in~\eqref{step1ag}, using~\eqref{mphg}, we obtain, for every $t\geq s+1$,
\begin{equation*}\label{}
\inf_{x\in B_{s}} \left( p\cdot x + \mm_\mu(te) - \mm_\mu(x+te) \right) \leq -\frac14\delta.
\end{equation*}
This contradicts the fact that $p=D\mm_\mu(e)$, since the latter implies, in view of the positive homogeneity of $\mm_\mu$, that 
 \begin{equation}\label{e.limqua}
\lim_{t\to \infty} \sup_{x\in B_s} \left| \mm_\mu(x+te) -  \mm_\mu(te) - p\cdot x \right| =0.
\end{equation}
This completes the proof of~\eqref{step1}.

 \smallskip

{\it Step 2.} We demonstrate that
\begin{equation}\label{step2}
\limsup_{\ep \to 0} \sup_{z\in B_R} w^\ep(z,p) \leq -\overline H(p).
\end{equation}
We may suppose that $\barH(p) > \barH_*$, since otherwise the claim follows from $\omega \in E_3$. 

The argument is similar to one introduced in~\cite{AS2}, relying on the limit~\eqref{mphg} and using $m_\mu$ as a supercorrector. Here it is a bit more simple then Step 1, since we do not need to use Straszewicz's theorem or to restrict our attention to exposed points of the sublevel set of $\barH$. Applying Lemma~\ref{geom} in view of~\eqref{e.pbndry} and the assumption that $\overline H(p) > \overline H_*$, we may select $e\in \partial B_1$ such that $p\in \partial\mm_{\mu}(e)$ and~$\mm_\mu(e)=e\cdot p$, where as usual we set $\mu:=\barH(p)$. The reason we do not need $p=D\mm_\mu(e)$ is because $m_\mu$ will be used as a supercorrector; so the fact that it may not be flat and rather ``bend upward" like a cone can only help in the comparison argument.

We consider a point $z\in B_s$ and $\ep,\delta > 0$ such that $w^{\ep}(z,p,\omega)+\barH(p) \ge \delta>0$. With $c>0$ and $\lambda<1$ to be selected, we consider the auxiliary function
\begin{equation}\label{}
W^\ep(x):=\lambda \left ( w^\ep(x,p) - w^\ep(z,p) \right )- c \delta \left( 1 + |x-z|^2 \right)^{\frac12}+ c \delta.
\end{equation}
Since $\omega\in E_2$, there exists $s>0$, which does not depend on $z$ or $\ep$,  such that 
\begin{equation}\label{}
U_\ep := \left\{ x\in \Rd \,:\, W^\ep(x) \geq -\frac14\delta \right\} \subseteq B_s(z).
\end{equation}
Choosing $\lambda$ sufficiently close to $1$ and $c>0$ sufficiently small depending on $\lambda$ and after similar computations arguments as in the demonstration of~\eqref{step1v}, we find that
\begin{equation}\label{}
 - \tr\left(A\left( \frac x\ep \right) D^2W^\ep \right) + H\left( p+DW^\ep,\frac x\ep \right) \leq \overline  H(p) - \frac12 \delta \quad \mbox{in} \ U_\ep. 
\end{equation}
The comparison principle yields
\begin{multline}\label{step2cp}
\inf_{ x\in U_\ep} \left( \ep m_\mu\left(\frac{x}{\ep},\frac{z-(s+1)e}{\ep}\right)-W^\ep(x)- p\cdot(x - z +(s+1)e)  \right) \\ = \inf_{x\in \partial U_\ep} \left(  \ep m_\mu\left(\frac{x}{\ep},\frac{z-(s+1)e}{\ep}\right)-W^\ep(x) - p\cdot(x-z+(s+1)e)  \right).
\end{multline}
Using that $W^\ep(z) = 0$ and $W^\ep \equiv -\delta/4$ on $\partial U_\ep \subseteq B_s(z)$ and rearranging, we obtain
\begin{equation} \label{step2ag}
 \inf_{x\in B_s(z)}\left( \ep  m_\mu\left(\frac{x}{\ep},\frac{z-(s+1)e}{\ep}\right)- \ep  m_\mu\left(\frac{z}{\ep},\frac{z-(s+1)e}{\ep}\right)-p\cdot (x-z)\right) \leq -\frac14\delta.
\end{equation}

To obtain a contradiction, we suppose that $w^{\ep_j}(z_j,p)+\barH(p) \ge \delta>0$ for sequences $\{ z_j\}_{j\in\N} \subseteq B_R$ and $\ep_j \to 0$. Applying~\eqref{step2ag} and sending $j\to \infty$ yields, in light of~\eqref{mphg}, 
\begin{equation}\label{}
\inf_{x\in B_{s}} \left(\mm_\mu(x+(s+1)e) - \mm_\mu((s+1)e) -  p\cdot x)\right) \leq -\frac14\delta.
\end{equation}
Since $\mm_\mu((s+1)e) =   (s+1)e\cdot p$ we conclude that, for some $x\in B_{s}$,
\begin{equation}\label{}
\mm_\mu(x+(s+1)e) -  p\cdot (x+(s+1)e)) \leq -\frac18\delta.
\end{equation}
This contradicts that $p\in \partial \mm_\mu(e)$ and finishes Step~2 and the proof of the Proposition. 
\end{proof}

\begin{remark}
The reader may object to the proof of Theorem~\ref{Hg} on the grounds that several steps in the proof are not as ``quantifiable" as promised in the introduction. In particular, it seems at first glance impossible to quantify (i) the limit in~\eqref{e.limqua} without extra information about the shape of the level sets of~$\barH$ (which is not easy to obtain) and (ii) Lemma~\ref{flatspot}, since it is obtained by a compactness argument.

About (i): this step is actually quantifiable because we can approximate the level sets of~$\barH$ by nice sets with positive curvature. Rather than the exposed points of the sublevel sets of~$\barH$, we may instead consider ``points of positive curvature" of the boundary of the level set, that is, points which also lie on the boundary of a large ball which contains the level set. The radius of this ball controls the rate of the limit~\eqref{e.limqua} and the error this introduces is relatively small. The details will appear in~\cite{AC}.

The second objection is more serious, but the phenomenon we encounter here is not artificial or a limitation of the method. Indeed, it was shown already in the first-order case~\cite{ACS} that the rate of convergence in the limit in~Lemma~\ref{flatspot} may be arbitrarily slow (even with a finite range of dependence quantifying the ergodicity assumption). In this sense, the proof above seems to optimally capture the underlying phenomena driving the homogenization of Hamilton-Jacobi equations in random media.
\end{remark}

\section{Homogenization: the proof of Theorem~\ref{Hg}}\label{H}

In this section we present the proof of our main result, Theorem~\ref{Hg}. The convergence result is obtained from the classical perturbed test function argument, suitably modified to handle the lack of uniform Lipschitz estimates for weakly coercive Hamiltonians. The argument can be seen as a method for showing that the homogenization result of~\eqref{macro-1}, which is a special case of~Theorem~\ref{Hg}, is actually strong enough to imply the theorem.

\smallskip

As in the previous section, we assume throughout that~$\P$ is a probability measure on~$(\Omega,\F)$ satisfying~\eqref{e.stat},~\eqref{e.erg} and~\eqref{e.wkcoer}.

\subsection{Well-posedness and basic properties}
Before giving the proof of homogenization, we first consider the question of well-posedness of solutions of the time-dependent initial-value problem
\begin{equation}\label{e.timeHJ}
\left\{ \begin{aligned}
& u^\ep_t -\ep \tr\left(A\left(\frac x\ep\right)D^2 u^\ep\right) + H\left( Du^\ep,\frac x\ep \right) = 0 & \mbox{in} & \ \Rd \times (0,\infty),\\
& u^\ep(\cdot,0) = g \in \BUC(\Rd). & &
\end{aligned} \right.
\end{equation}
For each $\ep > 0$, $g\in \BUC(\Rd)$ and $(x,t)\in \Rd\times(0,\infty)$, we define the random variable 
\begin{multline} \label{uepform2}
u^\ep(x,t,g) := \sup\Bigg\{ w(x,t) \,:\, w\in \USC(\Rd \times [0,t]) \, \mbox{is a subsolution of~\eqref{e.VHJ} in $\Rd \times [0,t)$,} \\ \limsup_{|x| \to \infty} \sup_{0<s\leq t} \frac{w(x,s)}{|x|} = 0 \ \ \mbox{and} \ \ w(\cdot,0)\leq g \ \mbox{on} \ \Rd \Bigg\}.
\end{multline}
This is the candidate for the (unique, we hope) solution of~\eqref{e.timeHJ}. Observe that we have 
\begin{equation} \label{ueplb}
u^\ep(x,t,g) \geq - \Lambda_1t + \inf_{\Rd} g,
\end{equation}
since the function on the right belongs to the admissible class in~\eqref{uepform2} by~\eqref{e.Hsubq} and~\eqref{e.sigbnd}. 

Similar to the situation for the approximate cell problem, checking that $(x,t) \mapsto u^\ep(x,t,g)$ does indeed solve~\eqref{e.timeHJ} reduces to proving a sublinear growth condition at infinity (uniformly in time). We remark that this is of interest only in the non-uniform coercivity case, since in this case well-posedness of~\eqref{e.timeHJ} is classical. 

\begin{lem}\label{GCt}
We have 
\begin{equation*} \label{}
\P \Big[ \forall  T>0, \ \forall g\in \BUC(\Rd),  \ \limsup_{|x| \to \infty} \sup_{0<t\leq T} \sup_{0<\ep\leq 1} \frac{|u^\ep(x,t,g)|}{|x|} = 0 \Big] = 1.
\end{equation*}
\end{lem}
\begin{proof}
In view of~\eqref{ueplb}, we may focus only on obtaining upper bounds for $u^\ep$. By definition, $g \mapsto u^\ep(x,t,g)$ is monotone nondecreasing and so we may suppose that $g$ is constant. Since $g\mapsto u^\ep(x,t,g)$ also commutes with constants, it suffices therefore to prove the sublinear growth estimate for $g\equiv 0$. That is, we need to show only the following:
\begin{equation*} \label{}
\P \Big[ \forall  T>0,   \ \limsup_{|x| \to \infty} \sup_{0<t\leq T} \sup_{0<\ep\leq 1} \frac{|u^\ep(x,t,0)|}{|x|} = 0 \Big] = 1.
\end{equation*}

We proceed by exhibiting an explicit supersolution and appealing to the comparison principle. The supersolution is
\begin{equation*} \label{}
V^\ep(x,t):= e^t w^\ep(x,0) + e^t \Lambda_1,
\end{equation*}
where $w^\ep(x,p)$ is, as in the previous section, the solution of~\eqref{macro}. The convexity of $H$ and~\eqref{e.Hsubq} imply that, for every $p\in \Rd$ and $\lambda \geq 1$, 
\begin{equation*} \label{}
\lambda^{-1} H(\lambda p,y) \geq H(p,y) - \left(1-\lambda^{-1} \right)H(0,y) \geq H(p,y) - \left(1-\lambda^{-1} \right) \Lambda_1.
\end{equation*}
Using this with $q=0$ fixed and $\lambda=e^t$, we find that, for each $t>0$, the function $w^\ep(\cdot,0)$ satisfies the inequality
\begin{equation*} \label{}
w^\ep - \ep \tr\left(A\left(\frac x\ep\right)D^2 w^\ep\right) + e^{-t} H\left( e^{t} Dw^\ep,\frac x\ep \right) \geq - \left( 1- e^{-t}\right) \Lambda_1  \quad \mbox{in}  \ \Rd.
\end{equation*}
From this it follows  that $V^\ep$ satisfies
\begin{equation*} \label{}
V^\ep_t - \ep\tr\left( A\left( \frac x\ep \right) D^2V^\ep \right) + H\left( DV^\ep,\frac x\ep \right) \geq 0 \quad \mbox{in} \ \R \times [0,\infty). 
\end{equation*}
Since $V^\ep$ is bounded below by $0$ uniformly in $\Rd \times [0,\infty)$, by comparing $V^\ep$ to any function in the admissible class in~\eqref{uepform2} using the comparison principle, we find that, for all $(x,t) \in \Rd \times [0,\infty)$ and every realization of the coefficients,
\begin{equation*} \label{}
u^\ep(x,t,0) \leq V^\ep(x,t). 
\end{equation*}
 According to~Lemma~\ref{GC}, 
\begin{equation*} \label{}
\P \Big[ \forall  T>0,   \ \limsup_{|x| \to \infty} \sup_{0<t\leq T} \sup_{0<\ep\leq 1} \frac{|V^\ep(x,t)|}{|x|} = 0 \Big] = 1.
\end{equation*}
This yields the lemma.
\end{proof}

As a consequence of~Lemma~\ref{GCt}, the lower bound~\eqref{ueplb} and the comparison principle (see~\cite[Theorem 2.3]{AT}), we obtain
\begin{multline*} \label{}
\P \Big[ \forall  \ep > 0, \forall g \in \BUC(\Rd), \ (x,t) \mapsto u^\ep(x,t,g) \ \mbox{belongs to $C(\Rd \times(0,\infty))$ and is the} \\ \mbox{unique solution of~\eqref{e.timeHJ} which, for all $T>0$, is bounded below on $\Rd \times [0,T)$} \Big] = 1.
\end{multline*}

\subsection{Homogenization}

In this subsection we complete the proof of Theorem~\ref{Hg}. Throughout we take $u(x,t,g)$ to be the unique solution of the homogenized problem
\begin{equation} \label{homgIVP2}
\left\{ \begin{aligned}
& u_t + \overline H( Du ) = 0 &  \mbox{in} & \ \Rd \times (0,\infty), \\
& u = g & \mbox{on} & \ \Rd \times \{ 0 \}.
\end{aligned} \right.
\end{equation}
In view of the growth condition~\eqref{Hgrth}, the problem~\eqref{homgIVP2} indeed possesses a unique solution, and it is given by the \emph{Hopf-Lax formula}
\begin{equation*} \label{}
u(x,t,g):= \inf_{y\in \Rd} \left( t\overline L\left( \frac{x-y}{t} \right) + g(y) \right)
\end{equation*}
where $\overline L :\Rd \to\R$ is the Legendre-Fenchel transform of $\overline H$, that is,
\begin{equation*} \label{}
\overline L(z):= \sup_{p\in\Rd} \left( p\cdot z - \overline H(p) \right).
\end{equation*}
Note that $\overline L$ is continuous, convex and satisfies $|z|^{-1}\overline L(z) \rightarrow +\infty$ as $|z| \to \infty$ (c.f.~Evans~\cite{Ebook}).

A proof that the Hopf-Lax formula defines a viscosity solution of~\eqref{homgIVP2} can be found for example in~Evans~\cite[Chapter 10]{Ebook} under the assumption that $g\in C^{0,1}_{\mathrm{loc}}(\Rd)$. It is  easy to extend this to the case that $g\in \BUC(\Rd)$ using the monotonicity of the Hopf-Lax formula in $g$ and the stability of viscosity solutions under local uniform convergence. The uniqueness of this solution follows from classical comparison principles for first-order equations. 

\smallskip

We now present the proof of the main result.

\begin{proof}[{\bf Proof of Theorem~\ref{Hg}}]
The theorem follows from Proposition~\ref{p.SH} by a variation of the classical perturbed test function argument first introduced by Evans~\cite{E2}. This comparison argument is entirely deterministic. The fact that the functions $u^\ep$ are not uniformly equi-Lipschitz continuous causes a technical difficulty which is overcome by the use of the parameter $\lambda$ in Step~1, an idea which first appeared in~\cite{AS1}.

To setup the argument, we let the events~$E_2$ and~$E_4$ be defined as in the proof of Proposition~\ref{p.SH} and set
\begin{multline*} \label{}
E_5:=  \Bigg\{ (\sigma,H) \in \Omega \,:\, \forall g\in \BUC(\Rd), \ \forall R>0, \\  \limsup_{\ep \to 0} \sup_{(x,t)\in B_R\times [0,R)} \left| u^\ep(x,t,g) - u(x,t,g) \right| = 0 \Bigg\}.
\end{multline*}
We claim that
\begin{equation} \label{e.PTFM}
E_2 \cup E_4  \subseteq E_5.
\end{equation}
Since $\P[E_2 \cap E_4] = 1$ by Lemma~\ref{GC} and Proposition~\ref{p.SH}, the theorem follows from~\eqref{e.PTFM}.

\smallskip

For the rest of the argument, we fix $(\sigma,H)\in E_2\cap E_4$, $g\in \BUC(\Rd)$ and $R>0$ and argue that 
\begin{equation*} \label{}
\limsup_{\ep \to 0} \sup_{(x,t)\in B_R\times [0,R)} \left| u^\ep(x,t,g) - u(x,t,g) \right| = 0.
\end{equation*}
By the comparison principle (c.f.~\cite{AT}), the flow $g\mapsto u^\ep(\cdot,t,g)$ is monotone nondecreasing as well as a contraction mapping on $L^\infty(\Rd)$. We may therefore assume without loss of generality that $g\in C^{1,1}(\Rd)$. For notational convenience we henceforth drop the dependence of $u$ and $u^\ep$ on $g$. 

We first argue that 
\begin{equation} \label{e.maintcl}
U(x,t):= \limsup_{\ep \to 0}u^\ep(x,t) \leq u(x,t). 
\end{equation}
By the comparison principle, it suffices to check that $U$ is a subsolution of the limiting equation and $U(\cdot,0) \leq g$. We handle these claims in the next two steps. 

\emph{Step 1.} To check that $U$ is a subsolution of the limiting equation, take a smooth test function $\psi\in C^\infty(\Rd\times(0,\infty))$ and a point $(x_0,t_0) \in\Rd \times (0,\infty)$ so that
\begin{equation}\label{Hg-1}
U-\psi \quad \text{has a strict local maximum at}\ (x_0,t_0).
\end{equation}
We must show that
\begin{equation}\label{Hg-goal}
\psi_t(x_0,t_0)+\barH(D\psi(x_0,t_0)) \le 0.
\end{equation}
Arguing by contradiction, we suppose on the contrary that 
\begin{equation}\label{Hg-2}
\eta := \psi_t(x_0,t_0)+\barH(D\psi(x_0,t_0)) >0.
\end{equation}
With $p_0:=D\psi(x_0,t_0)$ and $\lambda>1$ a constant to be selected below, we introduce the perturbed test function
\begin{equation*} \label{}
\psi^\ep(x,t):=\psi(x,t)+ \lambda w^\ep(x,p_0).
\end{equation*}
It is appropriate to compare $\psi^\ep$ to $u^\ep$, and to this end we must check that, for $\ep,r>0$ sufficiently small, $\psi^\ep$ is a solution of the inequality
\begin{equation} \label{Hg-3}
\psi^\ep_t -\ep \tr\left(A\left(\frac x\ep\right) D^2 \psi^\ep \right) + H\left( D\psi^\ep,\frac x\ep \right)\ge \frac16 \eta 
\quad \text{in} \ B\left(x_0,r\right)\times (t_0-r,t_0+r). 
\end{equation}

Let us admit the claim~\eqref{Hg-3} for the moment and show that it allows us to obtain the desired contraction, completing the proof that $U$ is a subsolution of the limiting equation. Applying the comparison principle (c.f.~\cite[Theorem~2.3]{AT}), in view of~\eqref{Hg-3} and the equation satisfied by~$u^\ep$, we deduce that 
\begin{equation*} \label{}
\sup_{B(x_0,r) \times (t_0-r,t_0+r)} \left( u^\ep-\psi^\ep \right) = \sup_{\partial \left(B(x_0,r) \times (t_0-r,t_0+r) \right)} \left( u^\ep-\psi^\ep \right).
\end{equation*}
This holds for all sufficiently small $r>0$ and $\ep > 0$, and by passing to the limit $\ep\to 0$, using that by $(\sigma,H) \in E_4$ we have that $w^\ep(\cdot,p_0)$ converges to the constant $-\overline H(p_0)$ uniformly on compact subsets of $\Rd$ as $\ep \to 0$, we find that
\begin{equation*} \label{}
\sup_{B(x_0,r) \times (t_0-r,t_0+r)}\left( U - \psi \right) = \sup_{\partial \left(B(x_0,r) \times (t_0-r,t_0+r) \right)} \left( U-\psi \right).
\end{equation*}
This holds for all sufficiently small $r>0$, which contradicts the assumption~\eqref{Hg-1}. 

\smallskip

To check that~\eqref{Hg-3} holds in the viscosity sense, we take a smooth test function $\varphi$ and a point $(x_1,t_1) \in B(x_0,r)\times (t_0-r,t_0+r)$ such that 
\begin{equation*} \label{}
\psi^\ep-\varphi \quad \mbox{has a strict local minimum at} \ (x_1,t_1). 
\end{equation*}
Rewriting this using the definition of $\psi^\ep$, we get
\begin{equation*} \label{}
(x,t) \mapsto w^\ep(x,p_0)-\lambda^{-1}(\varphi-\psi)(x,t) \quad \mbox{has a strict local minimum at} \ (x_1,t_1).
\end{equation*}
Using the equation for $w^\ep$, we find that
\begin{equation}\label{Hg-4}
w^\ep(x_1,p_0)-\ep \tr\left(A\left(\frac {x_1}\ep\right) \lambda^{-1}D^2 (\varphi-\psi)(x_1,t_1) \right) + H\left(p_0+ \lambda^{-1}D(\varphi-\psi),\frac {x_1}\ep \right)\ge 0.
\end{equation}
Using that  $(\sigma,H) \in E_4$ and $\psi$ is smooth, we may select $\ep>0$ sufficiently small and $\lambda$ sufficiently close to~$1$, so that
\begin{equation}\label{Hg-5}
\left|\lambda w^\ep(x_1,p_0)+\barH(p_0)\right|+\left | \ep \tr \left ( A\left ( \frac{x_1}\ep\right)D^2\psi(x_1,t_1) \right)  \right | \le \frac{\eta}{3}.
\end{equation}
Next, by selecting $r>0$ small enough, depending on $\lambda$ and $\psi$, we obtain
$$
(\lambda-1)^{-1}\left|\lambda p_0-D\psi(x_1,t_1)\right| \le |p_0|+(\lambda-1)^{-1}|p_0-D\psi(x_1,t_1)| \leq 2|p_0|.
$$
Using the convexity of $H$ together with the previous line and~\eqref{e.Hsubq}, we discover that
\begin{align*}\label{Hg-6}
\lefteqn{\lambda H\left(p_0+ \lambda^{-1}D(\varphi-\psi)(x_1,t_1),\frac {x_1}\ep \right)} \qquad & \\ 
& \leq H\left(D\varphi(x_1,t_1),\frac{x_1}\ep \right)+(\lambda-1)H\left(\frac{\lambda p_0-D\psi(x_1,t_1)}{\lambda-1},\frac{x_1}\ep \right) \nonumber\\
& \le  H\left(D\varphi(x_1,t_1),\frac{x_1}\ep \right) + \Lambda_1(\lambda-1) \left (2^m|p_0|^{\m}+1\right ). \nonumber
\end{align*}
Taking $\lambda>1$ closer to~1, if necessary, we obtain
\begin{equation} \label{Hg-6}
\lambda H\left(p_0+ \lambda^{-1}D(\varphi-\psi)(x_1,t_1),\frac {x_1}\ep \right) \leq  H\left(D\varphi(x_1,t_1),\frac{x_1}\ep \right)+ \frac{1}{3}\eta.
\end{equation}
Combining \eqref{Hg-4},~\eqref{Hg-5} and~\eqref{Hg-6} yields
\begin{equation}\label{Hg-7}
-\barH(p_0)-\ep\tr\left(A\left(\frac {x_1}\ep\right) D^2 \varphi(x_1,t_1) \right) + H\left(D\varphi,\frac {x_1}\ep \right)\geq -\frac{2}{3}\eta
\end{equation}
and then combining \eqref{Hg-2} and \eqref{Hg-7} gives
\begin{equation*} \label{}
\psi_t(x_0,t_0) -\ep\tr\left(A\left(\frac {x_1}\ep\right) D^2 \varphi(x_1,t_1) \right) + H\left(D\varphi,\frac {x_1}\ep \right)\geq \frac1{3}\eta.
\end{equation*}
By making $r>0$ smaller, if necessary, and using $\varphi_t(x_1,t_1) = \psi(x_1,t_1)$, we obtain 
\begin{equation*} \label{}
\varphi_t(x_1,t_1) -\ep\tr\left(A\left(\frac {x_1}\ep\right) D^2 \varphi(x_1,t_1) \right) + H\left(D\varphi,\frac {x_1}\ep \right)\geq \frac1{6}\eta.
\end{equation*}
This completes the proof of~\eqref{Hg-3} and thus that of Step 1. 

\smallskip

\emph{Step 2.} We next show that $U(\cdot,0) \leq g$ or, more precisely, that for every $R>0$, 
\begin{equation} \label{e.barrclm}
\limsup_{t\to 0} \sup_{x\in B_R} \left( U(x,t) - g(x) \right) \leq 0. 
\end{equation}
To accomplish this, we must construct supersolution barriers from above and apply the comparison principle. Note that this is very easy to do in the uniformly coercive case, we simply use the map $(x,t) \mapsto g(x) + k t$ where $k>0$ is a large constant depending on the constants in the hypotheses and $\| g \|_{C^{1,1}(\Rd)}$. Unfortunately, this function is not a supersolution in the nonuniformly coercive case, and so we need to consider a more elaborate barrier function. Rather than construct a barrier from scratch, we build it from the functions $w^\ep$ and use the fact that these homogenize. 

For each fixed $x_0\in \Rd$, the functions we consider have the form
\begin{equation*} \label{}
V^\ep(x,t):= 2W^{\ep}(x,t) - \phi(x,t),
\end{equation*}
where
\begin{equation*} \label{}
W^\ep (x,t) := e^t w^\ep\!\left(x,\tfrac12 Dg(x_0)\right) + \overline H(\tfrac12 Dg(x_0)) + \frac12 g(x_0) + \frac12 Dg(x_0) \cdot (x-x_0).
\end{equation*}
and 
\begin{equation*} \label{}
\phi(x,t):= - 2\left( 1 + \| g \|_{C^{1,1}(\Rd)} \right) \left( \left( 1+ |x-x_0|^2\right)^{1/2} -1 \right)- k\!\left(e^t-1\right).
\end{equation*}
and $k>0$ is a constant depending only on $g$, $x_0$, and other structural constants, defined by
\begin{equation*} \label{}
k:= 2\Lambda_2 \left( 1 + \| g \|_{C^{1,1}(\Rd)}\right) + \Lambda_1 \left( 2^\m\left( 1 + \| g \|_{C^{1,1}(\Rd)}\right)^\m + 1\right) + 2\Lambda_1\left( 2^{-\m} |Dg(x_0)|^\m+1 \right). 
\end{equation*}
We next derive an supersolution inequality for $W^\ep$. The convexity of $H$ and~\eqref{e.Hsubq} imply that, for every $p,\hat p\in \Rd$ and $\lambda \geq 1$, 
\begin{equation*} \label{}
\lambda^{-1} H(\lambda p +\hat p,y) \geq H(p +\hat p,y) - \left(1-\lambda^{-1} \right)H(\hat p,y) \geq H(p +\hat p,y) - \left(1-\lambda^{-1} \right) \Lambda_1\left( |\hat p|^\m+1 \right). 
\end{equation*}
Using this with $\hat p$ fixed and $\lambda=e^t$, we find that, for each $t>0$, the function $w^\ep(\cdot,\hat p)$ satisfies the inequality
\begin{equation*} \label{}
w^\ep - \ep \tr\left(A\left(\frac x\ep\right)D^2 w^\ep\right) + e^{-t} H\left( q+ e^{t} Dw^\ep,\frac x\ep \right) \geq - \left( 1- e^{-t}\right) \Lambda_1 \left( |\hat p|^\m+1 \right) \quad \mbox{in}  \ \Rd.
\end{equation*}
From this we see that $W^\ep$ satisfies the inequality
\begin{multline*} \label{}
W^\ep_t  -\ep \tr\left(A\left(\frac x\ep\right)D^2 W^\ep\right) + H\left( DW^\ep,\frac x\ep \right) \\ \geq -\left( e^t-1 \right) \Lambda_1 \left( 2^{-\m} \left|Dg(x_0)\right|^\m +1 \right) \quad \mbox{in}  \ \Rd \times (0,\infty).
\end{multline*}
On the other hand, we see by a routine calculation, using the definition of $k$,~\eqref{e.Hsubq} and~\eqref{e.sigbnd}, that $\phi$ is a (smooth) subsolution of the inequality 
 \begin{equation*} \label{}
\phi_t  -\ep \tr\left(A\left(\frac x\ep\right)D^2 \phi\right) + H\left( D\phi,\frac x\ep \right) \leq -2 e^t \Lambda_1 \left( 2^{-\m} \left|Dg(x_0)\right|^\m +1 \right) \quad \mbox{in}  \ \Rd \times (0,\infty).
\end{equation*}
The definition of $k$ has been split into three terms, and we see from~\eqref{e.Hsubq} that the first two terms take care of the contributions from spacial derivatives of $\phi$ and the third term is responsible for the right-hand side. 
 
We may now apply~\cite[Lemma 2.5 and Remark 2.6]{AT} to find that $V^\ep$ is a supersolution of
\begin{equation*} \label{}
V^\ep_t  -\ep \tr\left(A\left(\frac x\ep\right)D^2 V^\ep \right) + H\left( DV^\ep,\frac x\ep \right) \geq 0 \quad \mbox{in} \ \Rd \times (0,\infty). 
\end{equation*}
Therefore the comparison principle implies that, for every $\ep > 0$,
\begin{equation} \label{pinneddn}
u^\ep \leq V^\ep - \inf_{x\in\Rd} \left( V^\ep(x,0) - g(x) \right) \quad \mbox{in} \ \Rd \times [0,\infty).
\end{equation}

Since $w^\ep$ is bounded below (see~\eqref{wepbb}) and~$g$ is bounded, the linearly growing term in $\phi$ ensures that~$V^\ep(\cdot,0)$ is larger than $g$ outside a ball of fixed radius and centered at~$x_0$. But due to the fact that~$\omega=(\sigma,H)$ belongs to $E_4$, we have that, for every~$R>0$,
\begin{equation*} \label{}
\lim_{\ep \to 0} \sup_{x\in B_R} \sup_{0 \leq t\leq R} \left| V^\ep (x,t) - V  (x,t) \right| =0
\end{equation*}
where 
\begin{multline*} \label{}
V(x,t):= 2\left( e^t -1 \right) \overline H\left(\tfrac12 Dg(x_0) \right)+ g(x_0) + Dg(x_0) \cdot(x-x_0) \\ + 2\left( 1 + \| g \|_{C^{1,1}(\Rd)} \right)\left(  \left( 1+ |x-x_0|^2\right)^{1/2}-1\right) + k\!\left( e^t -1 \right).
\end{multline*}
It is routine to check that, for every $x\in \Rd$,
\begin{equation*} \label{}
 g(x) \leq  g(x_0) + Dg(x_0) \cdot(x-x_0) + 2\left( 1 + \| g \|_{C^{1,1}(\Rd)} \right)\left(  \left( 1+ |x-x_0|^2\right)^{1/2}-1\right) = V(x,0).
\end{equation*}
We deduce that 
\begin{equation*} \label{}
\limsup_{\ep \to 0} \inf_{x\in\Rd} \left( V^\ep(x,0) - g(x) \right) \geq 0. 
\end{equation*}
Since $V(x_0,0) = g(x_0)$ and $V$ is uniformly Lipschitz continuous on $\Rd \times [0,1)$ with a constant which is bounded above independently of $x_0$, this inequality combined with~\eqref{pinneddn} yields~\eqref{e.barrclm}.

\smallskip

\emph{Step 3.} 
We complete the proof by arguing that 
\begin{equation} \label{infbusiness}
\liminf_{\ep \to 0}u^\ep(x,t) \geq u(x,t). 
\end{equation}
The argument here is similar the demonstration of~\eqref{e.maintcl}. We omit the proof that the left side of~\eqref{infbusiness} is a supersolution of the limiting equation, since this part is essentially identical to~Step~1 (except that we remark that it is necessary to take $0<\lambda < 1$ in contrast to $\lambda > 1$ as we did above). The second step, which is the analogue of Step~2, is actually much easier because we may produce a single smooth function which is a subsolution of the heterogeneous equation for all~$\ep > 0$. Indeed, since~$H(p,x)$ is uniformly bounded above for bounded~$|p|$, we may take~$k> 0$ large enough, depending only on $\Lambda_1$, $\Lambda_2$ and $\| g \|_{C^{1,1}(\Rd)}$, such that $(x,t)\mapsto g(x) - kt$ is a subsolution of~\eqref{e.timeHJ}. Thus $u^\ep(x,t) \geq g(x) - kt$ for all $\ep > 0$, giving us the desired lower bound at the initial time. 
\end{proof}

\section{The proof of the quenched large deviations principle}
\label{QLD}

In this section we give the proof of Corollary~\ref{c.ldp} and study some properties of the rate function~$\barL$. The argument is due to Varadhan. 

Before giving the demonstration of Corollary~\ref{c.ldp}, let us see how the viscous Hamilton-Jacobi equation arises by considering the asymptotics of the partition function. According to the Feynman-Kac formula, for each $\omega\in \Omega$, the map $(x,t) \mapsto S(t,x,\omega)$ defined in~\eqref{e.Sdef} is a solution of the equation
\begin{equation*} \label{}
S_t - \tr\left(A(y,\omega) D^2S \right) - b(y,\omega)\cdot DS + V(y,\omega)S = 0 \quad \mbox{in} \ \Rd\times\R_+
\end{equation*}
and we have $S(0,\cdot,\omega) \equiv 1$. If we take the (inverse) Hopf-Cole transform of $S$, setting 
\begin{equation*} \label{}
U(x,t,\omega):= -\log S(t,x,\omega)
\end{equation*}
then we check that $(x,t) \mapsto U(x,t,\omega)$ is the unique viscosity solution of the initial-value problem
\begin{equation*} \label{}
\left\{ \begin{aligned}
& U_t - \tr\left(A(y,\omega) D^2U \right) + DU \cdot A(y,\omega) DU + b(y,\omega) \cdot DU - V(y,\omega) = 0 & \mbox{in} & \ \Rd\times \R_+, \\
& U(\cdot,0,\omega) \equiv 0 & \mbox{on} & \ \Rd.
\end{aligned} \right.
\end{equation*}
This suggests the definition~\eqref{e.HforLDP} of $H$. Rescale by setting 
\begin{equation} \label{e.uepU}
u^\ep(x,t,\omega):= \ep U\!\left(\frac x\ep,\frac t\ep,\omega\right), 
\end{equation}
and observe that $u^\ep$ is the solution of~\eqref{e.timeHJ} with $g \equiv 0$. An application of Theorem~\ref{Hg} yields
\begin{equation*} \label{}
\P \left[ \lim_{t\to \infty} \frac{1}{t} U(tx,t,\omega) = \lim_{\ep\to 0} u^\ep(x,1,\omega) = -\barH(0) \ \ \mbox{locally uniformly in} \ x\in \Rd \right]  = 1.
\end{equation*}
This gives the approximate likelihood that a particle survives for a very long time:   
\begin{equation} \label{e.Sasymp}
\sup_{|x|\leq Rt} e^{-\barH(0) t } S(t,tx,\omega) = \exp(o(t))  \quad \mbox{as} \ t\to \infty. 
\end{equation}
(Note that in this context we have $\barH(0) \leq 0$, as can be seen from the fact that $w^\ep \geq 0$ since the zero function is a subsolution of~\eqref{macro}.) In fact, we have just proved Corollary~\ref{c.ldp} in the case $K=U=\Rd$, since, by the duality of the Legendre transform,
\begin{equation*} \label{}
\inf_{y\in \Rd} \barL(y) = -\barH(0).
\end{equation*}

It turns out that by varying the initial condition $g$ in Theorem~\ref{Hg} (taking it to be approximately the characteristic function of $K$ or $U$) and using the Hopf-Lax formula for the solution of the limiting equation, this argument yields a proof of the large deviations principle. Here it is:

\begin{proof}[{\bf Proof of Corollary~\ref{c.ldp}}]
Fix an element $\omega \in \Omega$ belonging to the event inside the probability in the conclusion of Theorem~\ref{Hg}. We prove only the upper bound since the argument for the lower bound is similar. Select a positive, uniformly continuous function $g$ on $\Rd$ such that $g\leq 1$ in $\Rd$ and $g\equiv 1$ on $K$, and observe that
\begin{multline} \label{e.upp1}
- \log Q_{t,x,\omega} \left[ X_t \in sK\right]  \\ \geq \underbrace{-\log E_{x,\omega}\left[ g\left(X_t/s\right)\exp\left( -\int_0^t V(X_s,\omega)\, ds \right) \right]}_{=:U(x,t,\omega;s)} + \log S(t,x,\omega).
\end{multline}
The limit of the second term on the right side is given by~\eqref{e.Sasymp}:
\begin{equation*} \label{}
\lim_{t\to \infty} \frac1t \log S(t,tx,\omega) = \barH(0).
\end{equation*}
Therefore we concentrate on the first term on the right of~\eqref{e.Sasymp}. By the Feynman-Kac formula and an inverse Hopf-Cole change of variables, the function $U$ defined in~\eqref{e.upp1} is a solution of the initial-value problem
\begin{equation*} \label{}
\left\{ \begin{aligned}
& U_t - \tr\left(A(y,\omega) D^2U \right) + DU \cdot A(y,\omega) DU + b(y,\omega) \cdot DU - V(y,\omega) = 0 & \mbox{in} & \ \Rd\times \R_+, \\
& U(\cdot,0,\omega;s) = -\log g(\cdot/s) & \mbox{on} & \ \Rd.
\end{aligned} \right.
\end{equation*}
Rescale by introducing
\begin{equation*} \label{}
u^\ep(x,t,\omega):= \ep U\left(\frac x\ep , \frac t\ep,\omega ; \frac1\ep \right)
\end{equation*}
and notice that $u^\ep$ satisfies the rescaled equation
\begin{multline*}
u^\ep_t - \ep \tr\left(A\left(\frac x\ep,\omega\right) D^2u^\ep \right) + Du^\ep\cdot A\left(\frac x\ep,\omega\right) \! Du^\ep \\ + b\left(\frac x\ep,\omega\right) \cdot Du^\ep - V\left(\frac x\ep,\omega\right) = 0 \quad \mbox{in}  \ \Rd\times \R_+
\end{multline*}
with the initial condition $u^\ep(\cdot,0,\omega) = -\log g$ on $\Rd$.

Since $\omega$ belongs to the event in the conclusion of Theorem~\ref{Hg}, we have
\begin{equation*} \label{}
\lim_{t\to \infty} \frac 1t U\left(tx,t,\omega;t\right) = \lim_{\ep \to 0} u^\ep(x,1,\omega) = u(x,1),
\end{equation*}
where $u = u(x,t)$ is the unique solution of the deterministic problem
\begin{equation*} \label{}
\left\{ \begin{aligned}
& u_t + \barH(Du) = 0 & \mbox{in} & \ \Rd\times \R_+, \\
& u(\cdot,0) = -\log g & \mbox{on} & \ \Rd.
\end{aligned} \right.
\end{equation*}
According to the Hopf-Lax formula, we have
\begin{equation*} \label{}
u(x,t) = \inf_{y\in \Rd} \left( t\barL\left( \frac{x-y}{t} \right) - \log g(y) \right).
\end{equation*}
Combining the last few lines, we obtain
\begin{equation*} \label{}
\lim_{t\to \infty} \frac 1t U\left(tx,t,\omega;t\right) = \inf_{y\in \Rd} \left( \barL\left( x-y \right) - \log g(y) \right).
\end{equation*}
Inserting into~\eqref{e.upp1}, we obtain
\begin{equation*} \label{}
\lim_{t\to\infty} -\frac1t \log Q_{t,tx,\omega} \left[ X_t \in tK\right] \geq \inf_{y\in \Rd} \left( \barL\left( x-y \right) - \log g(y) \right) + \barH(0). 
\end{equation*}
Using the continuity of $\barL$ and taking $g$ to approximate the characteristic function of $K$, we obtain
\begin{equation*} \label{}
\lim_{t\to\infty} -\frac1t \log Q_{t,tx,\omega} \left[ X_t \in tK\right] \geq \inf_{y\in K}  \barL\left( x-y \right) + \barH(0). \qedhere\end{equation*}
\end{proof}

\subsection*{Acknowledgements}
We thank Stefano Olla and Ofer Zeitouni for comments and references. S.~Armstrong thanks the Forschungsinstitut f\"ur Mathematik (FIM) of ETH Z\"urich for support.

\bibliographystyle{plain}
\bibliography{HJhg}

\end{document}